\documentclass[11pt]{amsart}
\usepackage{amssymb,amsmath, amsthm, amsfonts}

\usepackage[usenames,dvipsnames]{xcolor}
\usepackage{graphicx}
\usepackage{listings}

\usepackage[normalem]{ulem}

\usepackage[margin=1in]{geometry}
\usepackage{lstautogobble}
\usepackage{enumerate}
\usepackage[shortlabels]{enumitem}
\usepackage{thmtools}
\usepackage{thm-restate}
\usepackage{amsthm}
\usepackage{verbatim}
\usepackage{accents}
\usepackage{mathtools}
\usepackage{physics}

\usepackage[colorlinks=true, allcolors=magenta,linkcolor=blue, citecolor=magenta]{hyperref}
\usepackage[capitalise]{cleveref}
\usepackage[bottom]{footmisc}
\crefformat{equation}{(#2#1#3)}

\usepackage{float}
\restylefloat{table}

\usepackage{mathrsfs}
\setlist{  
  listparindent=\parindent,
  parsep=0pt,
}

\theoremstyle{plain}
\newtheorem{thm}{Theorem}[section]
\newtheorem{prop}[thm]{Proposition}
\newtheorem{lemma}[thm]{Lemma}

\theoremstyle{definition}

\newtheorem{remark}[thm]{Remark}

\Crefname{thm}{Theorem}{Theorems}
\Crefname{prop}{Proposition}{Propositions}

\numberwithin{equation}{section} 

\DeclarePairedDelimiter{\pa}{\lparen}{\rparen}

\DeclareMathOperator{\supp}{supp}

\DeclareMathOperator{\diam}{diam}

\DeclareMathOperator{\dist}{dist}
\DeclareMathOperator{\sgn}{sgn}

\newcommand{\M}{{\mathcal{M}}}

\newcommand{\p}{{\partial}}

\newcommand{\cre}{\color{red}}

\renewcommand{\d}{\mathsf{d}}

\newcommand{\R}{{\mathbb{R}}}

\newcommand{\N}{{\mathbb{N}}}

\renewcommand{\H}{{\mathcal{H}}}
\newcommand{\T}{{\mathbb{T}}}
\newcommand{\g}{{\mathsf{g}}}

\newcommand{\Sc}{{\mathcal{S}}}

\renewcommand{\M}{{\mathbb{M}}}
\newcommand{\I}{\mathbb{I}}

\renewcommand{\k}{\mathsf{k}}
\newcommand{\ga}{\gamma}
\newcommand{\nab}{\nabla}

\newcommand{\wt}{\widetilde}
\newcommand{\tl}{\tilde}

\newcommand{\D}{\Delta}

\newcommand{\ph}{\phantom{=}}
\newcommand{\nn}{\nonumber}

\newcommand{\ol}{\overline}

\newcommand{\XN}{X_N}
\newcommand{\ZN}{Z_N}

\newcommand{\ux}{X}

\newcommand{\uz}{Z}
\newcommand{\ep}{\epsilon}
\newcommand{\vep}{\varepsilon}
\newcommand{\al}{\alpha}
\newcommand{\be}{\beta}
\newcommand{\ka}{\kappa}
\newcommand{\la}{\lambda}
\newcommand{\Om}{\Omega}

\newcommand{\indic}{\mathbf{1}}

\newcommand{\Fr}{\mathsf{F}}
\newcommand{\Hr}{\mathsf{H}}

\newcommand{\Uu}{\mathfrak{U}}

\newcommand{\E}{{\mathbb{E}}}

\newcommand{\Ec}{\mathcal{E}}

\newcommand{\cd}{\mathsf{c}_{\mathsf{d},\mathsf{s}}}

\newcommand{\s}{\mathsf{s}}
\newcommand{\w}{\mathsf{w}}
\renewcommand{\k}{k}

\renewcommand{\P}{\mathcal{P}}

\let\div\relax
\DeclareMathOperator{\div}{\mathrm{div}}

\def\XXint#1#2#3{{\setbox0=\hbox{$#1{#2#3}{\int}$ }
\vcenter{\hbox{$#2#3$ }}\kern-.6\wd0}}

\let\oldtocsection=\tocsection
 
\let\oldtocsubsection=\tocsubsection
 
\let\oldtocsubsubsection=\tocsubsubsection
 
\renewcommand{\tocsection}[2]{\hspace{0em}\oldtocsection{#1}{#2}}
\renewcommand{\tocsubsection}[2]{\hspace{1em}\oldtocsubsection{#1}{#2}}
\renewcommand{\tocsubsubsection}[2]{\hspace{2em}\oldtocsubsubsection{#1}{#2}}

\newcommand{\com}[1]{{\color{cyan}{*** #1 ***}}}

\title[The Lake equation as a supercritical mean-field limit]{The Lake equation as a supercritical mean-field limit}

\author[M. Rosenzweig]{Matthew Rosenzweig}
\address{Matthew Rosenzweig, Carnegie Mellon University, Department of Mathematical Sciences, Pittsburgh, PA} 
\email{mrosenz2@andrew.cmu.edu}
\thanks{M.R. was supported by the Simons Foundation through the Simons Collaboration on Wave Turbulence and by NSF grants DMS-2052651, DMS-2206085, DMS-2345533.}
\author[S. Serfaty]{Sylvia Serfaty}
\address{Sylvia Serfaty, Courant Institute of Mathematical Sciences, New York University, New York City, NY}
\email{serfaty@cims.nyu.edu}
\thanks{S.S. was supported by NSF grant DMS-2247846 and by the Simons Foundation through the Simons Investigator program.}

\begin{document}
\begin{abstract}
We study so-called supercritical mean-field limits of systems of trapped particles moving according to Newton's second law with either Coulomb/super-Coulomb or regular interactions, from which we derive a $\d$-dimensional generalization of the \emph{Lake equation}, which coincides with  the incompressible Euler equation in the simplest setting, for monokinetic data. This supercritical mean-field limit may also be interpreted as a combined mean-field and quasineutral limit, and our assumptions on the rates of these respective limits are shown to be optimal. Our work provides a mathematical basis for the \emph{universality} of the Lake equation in this scaling limit---a new observation---in the sense that the dependence on the interaction and confinement is only through the limiting spatial density of the particles. Our proof is based on a modulated-energy method and takes advantage of regularity theory for the  obstacle problem for the fractional Laplacian.

\end{abstract}
\maketitle

\section{Introduction}\label{sec:intro}
Consider a \emph{Newtonian system} of $N$ particles with a pairwise \emph{interaction potential} $\g$ and external \emph{confining potential} $V$:
\begin{equation}\label{eq:NewODE}
\begin{cases}
\dot{x}_i^t = v_i^t \\
\dot{v}_i^t = -\ga v_i^t \displaystyle-\frac{1}{\vep^2 N} \sum_{1\leq j\leq N: j\neq i}\nabla\g(x_i^t-x_j^t) -\frac{1}{\vep^2}\nabla V(x_i^t)\\
(x_i^0,v_i^0) = (x_i^\circ,v_i^\circ),
\end{cases}
\qquad 1\le i\le N.
\end{equation}
The positions and velocities are assumed to belong to $\R^\d$. Here, $\ga\geq 0$ is the \emph{friction coefficient}, and $\vep>0$ is a small parameter, possibly depending on $N$, which encodes physical information about the system. We are particularly interested in $\g$ belonging to the family of \emph{logarithmic} or \emph{Riesz} interactions
\begin{equation}\label{eq:gmod}
\g(x)= \begin{cases}  \frac{1}{\s} |x|^{-\s}, \quad  & \s \neq 0\\
-\log |x|, \quad & \s=0,
\end{cases}
\end{equation}
with the assumption that  $\d-2\le \s<\d$. Up to a normalizing constant $\cd$, these interactions are characterized as fundamental solutions of the fractional Laplacian: $(-\Delta)^{\frac{\d-\s}{2}}\g = \cd \delta_0$. The particular case $\s=\d-2$  corresponds to the classical \emph{Coulomb} interaction, which is fundamental to plasma physics, and thus $\s \ge \d-2$ means that we are considering the \emph{super-Coulomb} case. See \cref{rem:SCrest} below for further elaboration on the restriction to this case.

{We are interested in the large $N$ and small $\vep$ limit of \eqref{eq:NewODE}, the latter of which we interpret as a \emph{quasineutral limit} elaborated on in the next subsection. There are two (mathematically equivalent) motivations for our setup.

The first motivation is non-neutral plasmas \cite{OD1998} (see also \cite{WBIP1985,MKTL2008} for relevance to trapped neutral systems). The system \eqref{eq:NewODE} models the evolution of a trapped system of ions near thermodynamic equilibrium, meaning the spatial density $\mu_N^t \coloneqq \frac1N\sum_{i=1}^N\delta_{x_i^t}$ is close to the \emph{equilibrium measure} $\mu_V$. This equilibrium measure is defined as
the probability measure that minimizes the {macroscopic} energy
\begin{equation}\label{eq:Edef}
\mathcal{E}(\mu) \coloneqq \int_{\R^\d}Vd\mu + \frac{1}{2}\int_{(\R^\d)^{ 2}}\g(x-y)d\mu(x) d\mu(y).
\end{equation}
We refer to  the recent lecture notes of the second author \cite[Chap. 2]{SerfatyLN} for  details on the equilibrium measure in this context of Coulomb and Riesz interactions, as well as for its connection to the solution of the fractional obstacle problem, which will be also discussed below. 
One can see the equilibrium measure as a generalization of  the uniform measure on a torus when considering an infinitely extended trapped system.

The second motivation comes from two-species globally neutral systems. In this setting, the empirical spatial density $\mu_N^t$ is close to a fixed density $\mu$, representing the density of a stationary background of an oppositely charged species of particles (e.g.~heavy positively charged ions). Alternatively, one may think of this as a one-component plasma with a nonuniform background. In this case, the term $-\nab V(x_i^t)$ on the right-hand side of the second line of \eqref{eq:NewODE} should be replaced by the attractive force $+\nab(\g\ast\mu)(x_i^t)$ due to the background.

The settings for each of these motivations are mathematically equivalent. The latter corresponds to $V = -\g\ast\mu + \frac12\int_{(\R^\d)^2}\g(x-y)d\mu^{\otimes 2}(x,y)$. While the former corresponds to considering a background density such that $\nabla(\g\ast\mu+V)=0$ in the support of $\mu$. This is in particular achieved when (but not only when) $\mu$ is equal to the equilibrium measure $\mu_V$. 
}


Under suitable assumptions on the external potential $V$, our goal is  to show that if the initial \emph{empirical spatial density} $\mu_N^\circ \coloneqq \frac1N\sum_{i=1}^N \delta_{x_i^\circ} \rightharpoonup \mu_V$ as $N\rightarrow\infty$ and the initial velocities $v_i^\circ \approx u^\circ(x_i^\circ)$, for a macroscopic vector field $u^\circ$ on $\R^\d$, then the \emph{empirical measure} $\frac{1}{N}\sum_{i=1}^N \delta_{(x_i^t,v_i^t)}$ associated to a solution of \eqref{eq:NewODE} converges as $\vep\rightarrow 0$ and $N\rightarrow\infty$ to the \emph{monokinetic} measure $\mu_V(x)\delta_{u^t(x)}(v)$, where $u^t$ satisfies the \emph{Lake equation}
\begin{equation}\label{eq:Lake}
\begin{cases}
\p_t u +\ga u+ u \cdot\nabla u = -\nabla p\\
\div(\mu_V u) = 0.
\end{cases}
\end{equation}
In particular, when $\g$ is the Coulomb potential, we give a microscopic counterpart to the proof of the \emph{quasineutral} limit for Vlasov-Poisson with monokinetic data by  Barr\'{e} \emph{et al.} \cite{BCGM2015}.

Note that if $\mu_V$ is constant and $\ga=0$, then \eqref{eq:Lake} is nothing but the \emph{incompressible Euler equation}. The \emph{pressure} $p$ is a Lagrange multiplier to enforce the incompressibility constraint $\div(\mu_V u)=0$. Multiplying both sides of the first equation of \eqref{eq:Lake} by $\mu_V$ and taking the divergence, the pressure $p$ is obtained from the velocity $u$ by solving the divergence-form elliptic equation
\begin{equation}\label{eq:press}
-\div\pa*{\mu_V\nabla p} = \div^{2}\pa*{\mu_V u^{\otimes 2}}= \div\pa*{\mu_V u\cdot\nabla u}.
\end{equation}

Equation \eqref{eq:Lake}, which is also sometimes called the \emph{anelastic equation}, appears in the modeling of atmospheric flows \cite{OP1962, Masmoudi2007} and superconductivity \cite{CR1997, DS2018, Duerinckx2018} and has been mathematically studied in \cite{LOT1996, LOT1996phy, BCGM2015, Duerinckx2018}.  In particular, the second author and Duerinckx \cite{DS2018} have shown that the equation arises as a mean-field limit for Ginzburg-Landau vortices with pinning and forcing. {We also mention that the Lake equation has been shown \cite{Menard2023lake} to be a mean-field limit for a model of vortices in shallow water with varying topography \cite{Richardson2000}.}

Our proof rests on the powerful modulated-energy method introduced by the second author \cite{Serfaty2017} and developed in subsequent works \cite{Duerinckx2016, Serfaty2020,BJW2019edp, NRS2021} to treat first-order Hamiltonian or gradient flows for the family of Riesz energies. It takes advantage of new sharp estimates controlling the first variation along a transport of the modulated energy by the modulated energy itself, recently obtained by the authors \cite{RS2022}. To handle the general nonuniformity of the equilibrium measure and the presence of boundaries, we recast the equilibrium measure as a solution to an obstacle problem for the fractional Laplacian and take advantage of recent regularity theory.


\subsection{The combined mean-field and quasi-neutral limit}\label{ssec:introMFQN}
To see how the equation \eqref{eq:Lake} appears as a formal limiting dynamics for the empirical measure of \eqref{eq:NewODE}, we argue as follows.

Suppose that the parameter $\vep>0$ is fixed. Then a formal calculation (e.g. see \cite{Jabin2014}) reveals that if the initial empirical measure $f_{N,\vep}^\circ \coloneqq \frac1N\sum_{i=1}^N \delta_{z_i^\circ}$ converges to a sufficiently regular probability measure $f_{\vep}^\circ$ as $N\rightarrow\infty$, where $z_i^\circ\coloneqq (x_i^\circ,v_i^\circ)$, then the time-evolved empirical measure $f_{N,\vep}^t \coloneqq \frac{1}{N}\sum_{i=1}^N \delta_{z_i^t}$ converges as $N\rightarrow\infty$ to a solution $f_\vep^t$ of the \emph{Vlasov equation with friction} 
\begin{equation}\label{eq:Vlas}
\begin{cases}
\p_t f_\vep+v\cdot\nabla_x f_\vep-\frac{1}{\vep^2}\nabla(V+\g\ast\mu_\vep)\cdot\nabla_v f_\vep - \div_v(\ga vf)=0\\
\mu_\vep = \int_{\R^\d}df_\vep(\cdot,v) \\
f_\vep|_{t=0} = f_\vep^\circ,
\end{cases}
\qquad (t,x,v) \in \R \times (\R^\d)^2.
\end{equation}
We now seek to {formally}  derive the Lake equation \eqref{eq:Lake} from the Vlasov equation \eqref{eq:Vlas} in the limit as $\vep\rightarrow0$.

We recall that the regime under consideration is when the spatial density $\mu_\vep^t$ converges to the equilibrium measure $\mu_V$ as $\vep\rightarrow 0$ (this is an assumption). Decomposing the potential 
\begin{equation}
V+\g\ast\mu_\vep = (V+\g\ast\mu_V) + \g\ast(\mu_\vep-\mu_V),
\end{equation}
the fact that, {by characterization of the equilibrium measure}, $V+\g\ast\mu_V$ is constant on the support of $\mu_V$ (see \cref{sec:EMOP} for details) implies that
\begin{equation}
\nabla\pa*{V+\g\ast\mu_\vep} = \nabla\g\ast(\mu_\vep-\mu_V), \qquad x\in\supp \mu_V.
\end{equation}
Assuming that the renormalized electric potential difference $\frac{1}{\vep^2}\g\ast(\mu_\vep-\mu_V)$ has a weak limit $p$ as $\vep\rightarrow 0$, we see that the weak limit $f\coloneqq \lim_{\vep\rightarrow 0}f_\vep$ satisfies the equation
\begin{equation}\label{eq:KIE}
\begin{cases}
\p_tf+v\cdot\nabla_xf - \nabla p\cdot\nabla_vf -\div_v (\ga v f)=0\\
\mu_V = \int_{\R^\d}df(\cdot,v)\\
f|_{t=0} = f^\circ,
\end{cases}
\end{equation}
where $f^\circ$ is the weak limit of $f_\vep^\circ$.

Let us now define the \emph{current} $J(x)\coloneqq \int_{\R^\d}v\, df(x,v)$ associated to \eqref{eq:KIE}. Integrating both sides of the first equation in \eqref{eq:KIE} with respect to {$v$}, then using that the spatial density is equal to $\mu_V$ for all time, we find that {$\div J^t$ is constant in time}. To find an equation for $J$, let us differentiate inside the integral to obtain
\begin{equation}
\p_t J = \int_{\R^\d}v\pa*{-v\cdot\nabla_x f + \div_v\Big((\ga v + \nab p) f)}dv.
\end{equation}
Since the integration is with respect to $v$, we can pull out $\nabla_x$ to write
\begin{equation}
-\int_{\R^\d}v (v\cdot\nabla_x)fdv = -\div\int_{\R^\d}v^{\otimes 2}df(\cdot,v),
\end{equation}
where the divergence may be taken with respect to either rows or columns since the tensor is symmetric. Integrating by parts (assuming $f$ vanishes sufficiently rapidly as $|v|\rightarrow\infty$),
\begin{equation}
\int_{\R^\d}v^j\div_v\pa*{(\ga v +\nabla p)f}dv =  -\int_{\R^\d}\delta^{ij}\pa*{\ga v^i+\p_{i}p}df(\cdot,v) = -\ga J^j -\mu_V\p_{j}p.
\end{equation}
Therefore,
\begin{equation}\label{eq:Jeq}
\p_t J + \div\int_{\R^\d}v^{\otimes 2}df(\cdot,v) = -\ga J -\mu_V\nabla p.
\end{equation}
This equation is not closed in terms of $(J,p)$, since the second term on the left-hand side requires knowledge of the second velocity moment of $f$, {which in turn depends on third moment and so on} (this is the famous closure problem for moments of the Vlasov equation, e.g.~see \cite{Uhlemann2018}). But making the \emph{monokinetic} or \emph{``cold electrons''}\footnote{This terminology, which is common in the physics literature, stems from the fact that the temperature of the distribution is zero.} ansatz $f(x,v)=\mu_V(x)\delta(v-u(x))$, it follows that $J=\mu_V u$. Since $\mu_V$ is independent of time, substituting this identity into \eqref{eq:Jeq} yields
\begin{equation}\label{eq:LakeJ}
\begin{cases}
\mu_V\p_t u +\div(\mu_V u^{\otimes 2}) = -\mu_V\pa*{\ga u+\nabla p},\\
\div(\mu_V u) = 0.
\end{cases}
\end{equation}
Assuming that $\mu_V$ is positive on its support, we see from dividing by $\mu_V$ that \eqref{eq:LakeJ} is equivalent to  \smallskip \eqref{eq:Lake}.

The limit as $N\rightarrow\infty$ for fixed $\vep>0$ corresponds to the \emph{mean-field limit} of \eqref{eq:NewODE}. When $\g$ is Coulomb and $\ga=0$, the equation \eqref{eq:Vlas} is known as \emph{Vlasov-Poisson}. More generally, for $\g$ as in \eqref{eq:gmod}, the equation is called \emph{Vlasov-Riesz}. It is a difficult problem to derive the Vlasov-Poisson/Vlasov-Riesz equation directly from \eqref{eq:NewODE}. While the case of regular potentials $\g$ (e.g. globally $C^{1,1}$) \cite{NW1974, BH1977, Dobrushin1979, Duerinckx2021gl} or even just potentials with bounded force $\nabla\g$ \cite{JW2016} is understood, the Coulomb case in general remains out of reach except in dimension 1 \cite{Trocheris1986,Hauray2014}. The best results for singular potentials are limited to forces $\nabla\g$ which are square integrable at the origin \cite{HJ2007, HJ2015, BDJ2024} (which barely misses the two-dimensional Coulomb case) or are for Coulomb potentials with short-distance vanishing cutoff \cite{BP2016, Lazarovici2016, LP2017, Grass2021}.\footnote{If one adds noise to the velocity equation in \eqref{eq:NewODE}, corresponding to the \emph{Vlasov-Fokker-Planck} mean-field limit, then the two-dimensional Coulomb case has been recently achieved by Bresch \emph{et al.} \cite{BJS2022}.} Recently, the second author together with Duerinckx \cite[Appendix]{Serfaty2020} proved the mean-field limit for the Vlasov-Poisson equation--and more generally super-Coulomb Vlasov-Riesz\footnote{Combining the total modulated energy introduced in this work with the commutator estimate \cite[Proposition 4.1]{NRS2021}, one can extend this result to the sub-Coulomb Vlasov-Riesz equation as well.}---for monokinetic/cold initial data, for which the Vlasov-Poisson equation reduces to the \emph{pressureless Euler-Poisson equation}.

In the plasma physics setting of Vlasov-Poisson, the limit $\vep\rightarrow 0$ is called the \emph{quasineutral limit}, and the equation \eqref{eq:KIE}, in the case $\mu_V\equiv 1$, is called the \emph{kinetic incompressible Euler (KIE) equation} \cite{Brenier1989}. The inhomogeneous case of \eqref{eq:KIE} does not seem to have previously appeared in the literature. Nowhere in the above reasoning did we assume a specific form for $\g$ (e.g. Coulomb). This demonstrates a certain \emph{universality} of the KIE for this kind of singular limit, which appears to be a new observation. In this plasma physics setting, the distribution function $f$ models the evolution of electrons against a stationary background of positively charged ions. After a rescaling to dimensionless variables, the parameter $\vep$ corresponds to the \emph{Debye (screening) length} of the system, which is the scale at which charge separation in the plasma occurs. When the Debye length is much smaller than the length scale of observation, the plasma is said to be quasineutral, since it appears neutral to an observer. The rigorous justification of the quasineutral limit is a difficult problem  and has been studied in \cite{BG1994, Grenier1995, Grenier1996,Grenier1999,Brenier2000,Masmoudi2001, HkH2015, HkR2016, HkI2017, HkI2017jde, GpI2018, GpI2020sing}. We refer the reader to the survey \cite{GpI2020qn} and references therein for further discussion on the quasineutral limit.

The preceding formal calculations suggest that in the limit as $N +\vep^{-1}\rightarrow\infty$, which one can physically interpret as a combined mean-field and quasineutral limit, the empirical measure $f_{N,\vep}^t$ of the Newtonian system \eqref{eq:NewODE} converges to a solution $f^t$ of the KIE \eqref{eq:KIE}, which reduces to the Lake equation \eqref{eq:Lake} for monokinetic solutions. Thus, we expect that if the particle velocities $v_i^t \approx u^t(x_i)$, then the empirical measure $f_{N,\vep}^t$ converges to the measure $\delta_{u^t(x)}(v)\mu_V(x)$ as $N+\vep^{-1}\rightarrow\infty$, where $u^t$ is a solution of \eqref{eq:LakeJ}. A more general interpretation of the limit as $N +\vep^{-1}\rightarrow\infty$ is as a \emph{supercritical mean-field limit} of the system \eqref{eq:NewODE}. This terminology coined by Han-Kwan and Iacobelli \cite{HkI2021} refers to the fact that the force experienced by a single particle in \eqref{eq:NewODE} formally diverges as $N\rightarrow\infty$, compared to being $O(1)$ for the usual $1/N$ mean-field scaling. 

\subsection{Prior work}\label{ssec:introPW}

The convergence of the empirical measure to $\delta_{u^t(x)}(v)\mu_V(x)$ was previously shown in the spatially periodic Coulomb case (i.e.~$x\in\T^\d$) when $V=0$ and $\mu_V\equiv 1$ by Han-Kwan and Iacobelli \cite{HkI2021} assuming $\vep N^{\frac{1}{\d(\d+1)}} \rightarrow \infty$ as $\vep+N^{-1}\rightarrow 0$, where $u^t$ is a solution of the incompressible Euler equation. In this work, they recognized (building on the aforementioned work \cite[Appendix]{Serfaty2020} for the mean-field limit) that the modulated-energy method may be also used for this supercritical limit provided one adds a suitable $O(\vep^2)$ corrector to the background spatial density in the definition of the modulated potential energy. We elaborate more on this idea in \cref{ssec:introMPf} below. Their proof may be viewed as a generalization of Brenier's \cite{Brenier2000} modulated-energy approach to proving the quasineutral limit of Vlasov-Poisson with monokinetic data to allow (via renormalization) for the solution of Vlasov-Poisson to be a sum of Diracs.

As in all modulated-energy approaches, the key tool is a functional inequality (see \cref{prop:comm} below) that controls the derivative of the modulated potential energy along a transport by the modulated potential energy itself:
\begin{align}\label{eq:introcomm}
\frac12\int_{(\R^\d)^2\backslash \triangle} (v(x)-v(y)) \cdot \nabla \g(x-y) d ( \mu_N- \mu)^{\otimes 2} (x,y) \le C_v(\Fr_N(\XN,\mu) + CN^{-\al})
\end{align}
for a constant $C_v$ depending quantitatively on the vector field $v$ and some $\alpha>0$. Such a functional inequality is, in fact, a type of \emph{commutator} estimate; namely, the quadratic form associated to the commutator $v\cdot \nab \g\ast - \div(\g\ast\cdot)$, ignoring the excision of the diagonal. Specifically, \cite{HkI2021} used the inequality of \cite{Serfaty2020}, whose non-sharp $O(N^{\frac{1}{\d(\d+1)}})$ additive error leads to the aforementioned restriction on $\vep$. 

In the same setting, this convergence was subsequently improved by the first author \cite{Rosenzweig2021ne} to
\begin{equation}\label{eq:RoseSA}
\lim_{\vep+N^{-1}\rightarrow 0} \vep^{-2}N^{-\frac{2}{\d}} =0
\end{equation}
using a sharp commutator estimate for the Coulomb case \cite{LS2018, Serfaty2023, Rosenzweig2021ne}. Sharper estimates in terms of the dependence on the solution $u$ of Euler's equation were also shown. Moreover, in \cite{Rosenzweig2021ne}, it was conjectured that the scaling assumption \eqref{eq:RoseSA} should be in general optimal, in the sense that the incompressible Euler equation should not be the limiting evolution of the empirical measure when $\vep^{-2} \ll N^{-\frac2\d}$.

These previous works are limited to the Coulomb case on the torus, which is an idealized setting since it assumes a spatially uniform density. Moreover, they left open the question of a rigorous justification of the universality of the Lake equation with respect to the interaction and confinement in the sense that it only depends on the equilibrium measure. Finally, these previous works also left open the optimality of the scaling relation between $\vep$ and $N$. As we now explain, we settle these questions for (super-)Coulomb Riesz interactions.


\subsection{Informal statement of main results}\label{ssec:introMR}

To present our result, we introduce the \emph{total modulated energy}
\begin{equation}\label{eq:soMEdef}
\Hr_{N}(\uz_N^t, u^t) \coloneqq \frac{1}{2N}\sum_{i=1}^N |v_i^t-u^t(x_i^t)|^2 + \frac{1}{\vep^2}\Fr_N(\ux_N^t,\mu_V+\vep^2\Uu^t)  + \frac{1}{\vep^2 N}\sum_{i=1}^N \zeta(x_i^t).
\end{equation}
Here, $\uz_N^t$ is a solution of the $N$-particle system \eqref{eq:NewODE}. The vector field $u^t$ is not a solution of the Lake equation \eqref{eq:Lake} but rather an extension of a solution from $\supp \mu_V $ to all of $\R^\d$, such that the regularity is preserved and $\div(\mu_V u) = 0$ (see \cref{sec:Lake}). The physical Cauchy problem for \eqref{eq:Lake} is in the domain $\supp\mu_V$ subject to a no-flux boundary condition, but at the microscopic dynamics level \eqref{eq:NewODE}, the particles are not confined to $\supp\mu_V$. The extension allows us to compare between the two settings. 

The first term of \eqref{eq:soMEdef} is the \emph{modulated kinetic energy}. The second term is the \emph{modulated potential energy} { introduced in \cite{SS2015,PS2017} and first used in the dynamical setting in \cite{Duerinckx2016, Serfaty2020},}  where
\begin{align}\label{eq:introMPEdef}
\Fr_N(X_N^t,\mu_V+\vep^2\Uu^t) \coloneqq \frac12\int_{(\R^\d)^2\setminus\triangle}\g(x-y)d\Big(\frac1N\sum_{i=1}^N\delta_{x_i^t} - \mu_V - \vep^2\Uu^t\Big)^{\otimes 2}(x,y),
\end{align}
{and $\Uu^t$ is an $N$-independent \emph{corrector}, whose precise definition is provided in \eqref{eq:Uudef}.  
Thanks to a certain {\it electric formulation}, this modulated energy has good coercivity properties and controls  a form of distance between the empirical measure $\frac1N \sum_{i=1}^N \delta_{x_i^t} $ and $\mu_V+\vep^2 \Uu$. These properties are reviewed in  \cite[Chap. 4]{SerfatyLN}), we will here quote the ones we need in Section \ref{sec:ME}.}

The third term is
\begin{align}\label{eq:introzetadef}
\zeta\coloneqq \g\ast\mu + V - c
\end{align}
for Robin constant $c$ (see \cref{ssec:EMOPem} for further elaboration), {which is nonnegative and vanishes on $\supp \mu_V$}.  Strictly speaking, the quantity \eqref{eq:soMEdef} depends on both $\vep,N$. But since we view $\vep$ as a function of $N$, we omit the dependence on $\vep$ to lighten the notation.

An informal statement of our main result is given in \cref{thm:mainSMF} below. In \cref{sec:MP}, we will give a precise statement (see \cref{thm:mainSMFrig}), in particular clarifying the vague assumptions,  after we have reviewed necessary background facts.  We also prove a generalization of our findings to the case of \emph{regular} interactions in \cref{sec:appreg}; but for brevity, we omit discussion of this result from the introduction.

\begin{thm}[Informal]\label{thm:mainSMF}
Suppose that the equilibrium measure $\mu_V$ is sufficiently regular and $\Sigma\coloneqq\supp  \mu_V$ {is a sufficiently smooth domain}. Assume $u$ is a sufficiently regular solution to \eqref{eq:Lake} on $[0,T]$. Then there exist continuous functions $C_1,C_2,C_3: [0,T]\rightarrow \R_+$, which depend on $\d,\s, \ga$, and norms of $u$, such that for any solution $Z_N^t = (X_N^t, V_N^t)$ of \eqref{eq:NewODE}, it holds that
\begin{align}\label{eq:mainSMF}
\Hr_{N}(\uz_N^t, u^t) + \frac{\log N}{2\d N\vep^2}\indic_{\s=0} \leq e^{C_1^t}\Big(\Hr_{N}(\uz_N^0, u^0)   +\frac{\log N}{2\d N\vep^2}\indic_{\s=0} + \frac{C_2^tN^{\frac{\s}{\d}-1}}{\vep^2} {+ C_3^t\vep^2} \Big).
\end{align}
In particular, if 
\begin{equation}\label{eq:SMFscl}
\lim_{\vep+\frac{1}{N}\rightarrow 0} \Big(\Hr_{N}(\uz_N^0,u^0) + \frac{\log N}{2\d N\vep^2}\indic_{\s=0}\Big) =  \lim_{\vep+\frac{1}{N}\rightarrow 0}  \frac{N^{\frac{\s}{\d}-1}}{\vep^2} =0,
\end{equation}
then
\begin{equation}\label{eq:SMFweak}
\frac1N\sum_{i=1}^N\delta_{z_i^t} \xrightharpoonup[\vep+\frac{1}{N}\rightarrow 0]{} \delta_{u^t(x)}(v)\mu_V(x), \qquad \forall t\in [0,T]
\end{equation}
in the weak-* topology for measures.
\end{thm}

The scaling assumption $N^{\frac{\s}{\d}-1}/\vep^2\rightarrow 0$ in \eqref{eq:SMFscl} is in general optimal, in the sense that there exists a sequence of solutions $\uz_N^t$ to \eqref{eq:NewODE} such that ${\frac1N\sum_{i=1}^N \delta_{z_i^t}} \xrightharpoonup[]{} \delta_{u^t(x)}(v)\mu_V(x)$ as $\vep+\frac1N\rightarrow 0$, but the total modulated energy $\Hr_{N}(\uz_N^t,u^t)$ does not vanish. This is a consequence of the next-order asymptotics for mean-field limits of log/Coulomb/Riesz energies obtained in \cite{SS2015,SS2015log,RS2016,PS2017}, comprehensively reviewed in  \cite{SerfatyLN}.

More precisely, suppose that $\ux_N^\circ$ is a minimizer of the microscopic energy
\begin{equation}\label{eq:ENdef}
\mathcal{H}_N(\ux_N) \coloneqq \frac{1}{2N}\sum_{1\leq i\neq j\leq N} \g(x_i-x_j) + \sum_{i=1}^N V(x_i).
\end{equation}
By taking variations, we see that $\ux_N^\circ$ is a critical point, i.e. 
\begin{equation}
\forall 1\le i\le N, \qquad \frac{1}{N}\sum_{1\leq j\leq N: j\neq i}\nabla\g(x_i^\circ-x_j^\circ) + \nabla V(x_i^\circ) = 0.
\end{equation}
For each $1\le i\le N$, define
\begin{equation}
\forall t\ge 0, \qquad (x_i^t, v_i^t) \coloneqq (x_i^\circ,0),
\end{equation}
so that  $\uz_N^t\coloneqq (x_i^t,v_i^t)_{i=1}^N$ is evidently the unique (stationary) solution of \eqref{eq:NewODE} with $\ga=0$, which is moreover \emph{independent of $\vep$}.
As a consequence of the results of \cite{SS2015,SS2015log,RS2016,PS2017}, it holds that $\frac{1}{N}\sum_{i=1}^N \delta_{x_i^\circ} \xrightharpoonup[N\rightarrow\infty]{} \mu_V$ and
\begin{equation}
\frac{1}{N}\sum_{i=1}^N\underbrace{\zeta(x_i^\circ)}_{=0} + \Fr_N(\ux_N^\circ, \mu_V) + \frac{\log N}{2\d N}\indic_{\s=0} = \mathsf{C}_{\d,\s}^VN^{\frac{\s}{\d}-1} + o(N^{\frac{\s}{\d}-1}) \quad \text{as $N\rightarrow\infty$},
\end{equation}
where $\mathsf{C}_{\d,\s}^V$ is a computable constant depending only on $\d,\s,V$ {that encodes thermodynamic information at the microscale.} That $\zeta(x_i^\circ)=0$ for each $i$ follows from the fact that minimizing point configurations lie in the support of the equilibrium measure, on which $\zeta$ vanishes (see \cite[Theorem 5]{PS2017}). 

 Hence,
\begin{equation}
\Hr_{N}(\uz_N^t, 0)+  \frac{\log N}{2\d N\vep^2}\indic_{\s=0} = \frac{1}{\vep^2}\pa*{\Fr_N(\ux_N^\circ, \mu_V)+ \frac{\log N}{2\d N}\indic_{\s=0}} =  \mathsf{C}_{\d,\s}^V\frac{N^{\frac{\s}{\d}-1}}{\vep^2} + o(\frac{N^{\frac\s\d-1}}{\vep^2}).
\end{equation}
Thus, 
we cannot expect vanishing of the total modulated energy in the supercritical mean-field regime if $\vep \leq N^{\frac{\s-\d}{2\d}}$.

 The observations of the preceding paragraph lead one to ask what is the effective equation describing the system \eqref{eq:NewODE} as $\vep+\frac{1}{N}\rightarrow 0$, assuming that $\vep \leq N^{\frac{\s-\d}{2\d}}$. Naively, one might expect that when $\vep \ll N^{\frac{\s-\d}{2\d}}$, this behaves like first sending $\vep\rightarrow 0$ and then $N\rightarrow\infty$. However, this limit does not make sense in general. Indeed, multiplying both sides of the second equation of \eqref{eq:NewODE} by $\vep^2$ and letting $\vep\rightarrow 0$ for fixed $N$, we formally see that the limiting positions  $\ux_{N}^t$ should be a critical point of the energy \eqref{eq:ENdef}. If each limiting velocity $v_{i}^t$ is nonzero, we would not expect $\ux_{N}^t$ to remain a critical point for all $t$. On the other hand, there is no mechanism to force the velocities to tend to zero even if the initial positions are a critical point of the energy \eqref{eq:ENdef}. We give an explicit demonstration of the failure of this limit for the exactly solvable one-dimensional Coulomb case in \cref{sec:1DCou} precisely when $\vep^{-2}N^{\frac{\s}{\d}-1} = (\vep N)^{-2}$ does not vanish, showing that there need not be any weak limit for the empirical measure, even in the simplest setting.  

Let us close this subsection with some further remarks about \cref{thm:mainSMF}.

\begin{remark}\label{rem:initMEvan}
Given $u^\circ$, one can produce statistically generic examples of initial data $\uz_N^\circ$ such that \eqref{eq:SMFscl} holds. Indeed, suppose that $(x_i^\circ)_{i=1}^\infty$ is a sequence of iid random points in $\R^\d$ with law $\mu_V$. For each $i\ge 1$, choose $v_i^\circ \in B(u^\circ(x_i^\circ), r_N)$, where the radius $r_N\rightarrow 0$ as $N\rightarrow\infty$. Then it is easy to check that for $\uz_N^\circ = (z_i^\circ)_{i=1}^N$ with $z_i^\circ = (x_i^\circ,v_i^\circ)$,  we have
\begin{multline}\label{eq:initMEvan}
\E\Big(\Hr_{N}(\uz_N^\circ, u^\circ) + \frac{\log N}{2\d N \vep^2}\indic_{\s=0} \Big)= {O(r_N^2)}  - \frac{1}{2\vep^2 N}\int_{(\R^\d)^2}\g(x-y)d(\mu_V)^{\otimes2}(x,y) + \frac{\log N}{2\d N \vep^2}\indic_{\s=0} \\
 + \frac{1}{\vep^2}\int_{\R^\d}\zeta d\mu_V.
\end{multline}
Since $\zeta$ is identically zero on the support of $\mu_V$, the last term vanishes. Evidently, the remaining terms on the right-hand side tend to zero as $\vep+\frac{1}{N}\rightarrow 0$ assuming that $\vep^2N\rightarrow\infty$ and $\s\ne 0$.  In the case $\s=0$, we need the distribution of $\XN^\circ$ to be sufficiently correlated so that
\begin{align}
\frac{1}{\vep^2}\E\Big(\Fr_{N}(\XN,\mu_V) + \frac{\log N}{2\d N}\indic_{\s=0}\Big)
\end{align}
vanishes as $\vep+\frac1N\rightarrow 0$ assuming that $\vep^2N \rightarrow\infty$. For this, it suffices to take $\XN^\circ$ to be distributed according to a modulated Gibbs measure
\begin{align}
d\mathbb{Q}_{N,\be}(\mu_V) \coloneqq \frac{1}{\mathsf{K}_{N,\be}(\mu_V)}e^{-\be N\Fr_N(\XN,\mu_V)}d\XN, \qquad \be \gtrsim N,
\end{align}
which corresponds to a (low-temperature) log gas. We refer to \cite[Chap. 5 and 6]{SerfatyLN} for details.
\end{remark}


\begin{remark}
By adapting ideas from our prior work \cite{RS2021} on first-order mean-field limits with additive noise, we expect that one can generalize our result to treat the \emph{Langevin system} with vanishing noise, so that the system \eqref{eq:NewODE} of ODEs is instead now a system of SDEs
\begin{align}\label{eq:NewSDE}
\begin{cases}
d{x}_i^t = v_i^tdt \\
d{v}_i^t = -\ga v_i^tdt  \displaystyle-\frac{1}{\vep^2 N} \sum_{1\leq j\leq N: j\neq i}\nabla\g(x_i^t-x_j^t)dt -\frac{1}{\vep^2}\nabla V(x_i^t)dt + \sqrt{2/\be}dW_i^t,
\end{cases}
\end{align}
where $W_1,\ldots,W_N$ are independent standard $\d$-dimensional Wiener processes and the differential is in the It\^o sense. In this case, the noise models thermal fluctuations at the microscopic level and the parameter $\be\geq 0$ has the interpretation of \emph{inverse temperature} and may depend on $N$.

 If one runs the same derivation as sketched in \cref{ssec:introMFQN} starting from \eqref{eq:NewSDE}, then the KIE \eqref{eq:KIE} now has an additional term $\be^{-1} \Delta_v f$ in the right-hand side of the first equation. The resulting equation makes mathematical sense, although we are not away of its study in the literature. However, the monokinetic ansatz is in general not compatible with this diffusion term. To understand why, observe that $\int_{\R^\d}v\Delta_v f dv = 0$ by integration by parts, so there is no contribution to the equation \eqref{eq:Jeq} for the current $J$; but by It\^o's formula, the noise has an order 1 contribution to the kinetic energy at the microscopic level, which does not vanish in the limit. To rectify this issue, we have to require that $\be=\be_N \rightarrow\infty$ as $N\rightarrow\infty$, so that the contribution of the noise vanishes in the limit.  Thus, thermal fluctuations vanish as $N\rightarrow\infty$, and the limiting equation is still \eqref{eq:Lake}. 
\end{remark}

\subsection{Method of proof}\label{ssec:introMPf}

The quantity \eqref{eq:soMEdef} is a variant of the total modulated energy originally introduced by Duerinckx and the second author \cite[Appendix]{Serfaty2020} to treat the mean-field limit for Vlasov-Riesz in the monokinetic regime. The idea to incorporate a time-dependent corrector $\Uu^t$ in the modulated potential energy for supercritical mean-field scalings originates in the aforementioned work of Han-Kwan and Iacobelli \cite{HkI2021}. The scaling by $\vep^2$ in the expression of $\mu_V+\vep^2\Uu^t$ reflects $O(\vep^2)$ fluctuations around the macroscopic equilibrium spatial density $\mu_V$. The addition of the last term in \eqref{eq:soMEdef} is a new contribution of the present work and reflects the fact that our starting system \eqref{eq:NewODE} is confined by an external potential $V$, as opposed to making an \emph{a priori} assumption that the domain of the problem is compact, e.g. $\T^\d$ as in \cite{HkI2021, Rosenzweig2021ne}. Although the $\zeta$ term appears to be only $O(1)$, it is, in fact, zero if the particles remain in the support of $\mu_V$, i.e. the quasineutral assumption is propagated. In analogy to the relationship between \cite{HkI2021} and \cite{Brenier2000}, our total modulated energy \eqref{eq:soMEdef} may be viewed as a renormalization of the total modulated energy from \cite{BCGM2015}, so as to allow for the Vlasov solution $f_\vep^t = \frac1N\sum_{i=1}^N \delta_{z_i^t}$. 

As with all modulated-energy approaches, our proof (see \cref{sec:MP} for the main argument) is based on establishing a Gr\"{o}nwall relation for the total modulated energy \eqref{eq:soMEdef}. The time derivative of this quantity has several terms that require different consideration (see \cref{lem:MEsmfID}). 

The main contribution from the modulated kinetic energy is trivially estimated using Cauchy-Schwarz. The main contribution from the modulated potential energy is a commutator of the form of the left-hand side of \eqref{eq:introcomm} with $v=\tl{u}^t$, the extension of the solution $u^t$ of the Lake equation to the whole space (see \cref{sec:Lake} for further elaboration), and $\mu=\mu_V+\vep^2\Uu^t$. To handle this term, we crucially rely on the authors' recent sharp estimate \cite[Theorem 1.1]{RS2022}, recorded in \cref{prop:comm} below, which is what enables us to achieve the scaling assumption \eqref{eq:SMFscl}. 

Another term,  new compared to \cite{HkI2021, Rosenzweig2021ne} and coming from the contribution of the $\zeta(x_i^t)$, is
\begin{align}
\frac1N\sum_{i=1}^N u^t(x_i^t)\cdot\nab\zeta(x_i^t) = \int_{\R^\d} u\cdot\nab\zeta d\mu_N^t.
\end{align}
This term has no commutator structure. Instead, we manage to bound it by $C\|u^t\|_{W^{1,\infty}}\frac1N\sum_{i=1}^N \zeta(x_i^t)$ thanks to \cref{lem:ugradVext}. The proof of this lemma relies on the fact that $u^t$ satisfies a no-flux condition on the boundary of $\Sigma = \supp\mu_V$ (a consequence of taking the quasineutral limit) and some nontrivial results for the regularity of the free boundary for solutions of the obstacle problem for the fractional Laplacian (see \cref{sec:appOP}), which may be of independent interest.  We mention that a similar term was encountered in \cite{BCGM2015} in the Coulomb case, where the $\mu_N^t$ is replaced by the spatial density $\mu^t$ of Vlasov-Poisson, but handled by \emph{ad hoc} arguments {using the local nature of the Laplacian}.

The correction $\vep^2\Uu^t$ in the spatial density is to cancel out the contribution of the pressure in \eqref{eq:Lake} when one differentiates the modulated kinetic energy. The exact definition of $\Uu^t$ is given in \eqref{eq:Uudef}, and we refer to \eqref{eq:dtHcon} and \eqref{eq:nabhUu} for the exact cancellation. 

There are also several residual terms that have the form $\int_{\R^\d} \phi\,  d(\mu_N-\mu_V-\vep^2\Uu)$, where $\phi$ is a function of $u,\Uu$. These may be controlled by the modulated potential energy, thanks to its coercivity (see \cref{lem:MPEcoer}), plus errors which are $O(\vep^{-2}N^{\frac{\s}{\d}-1})$, hence acceptable.

Combining the estimates for the various terms and appealing to the Gr\"onwall-Bellman lemma yields the inequality \eqref{eq:mainSMF}. As explained in {\cite[Section 4.2]{Rosenzweig2021ne}}, the weak convergence of the empirical measure follows from the vanishing of the total modulated energy using the coercivity of the modulated potential energy. This then completes the proof of \cref{thm:mainSMF}.

\begin{remark}\label{rem:SCrest}
We have considered only the Coulomb/super-Coulomb sub-case of log/Riesz interactions for two reasons. First, this is the only case where we can show our results are sharp because of the sharpness of our commutator estimates. Only a non-sharp commutator estimate is available in the sub-Coulomb case by work of the authors and Q.H. Nguyen \cite{NRS2021}. Second, and more importantly, there does not seem to be an adequate regularity theory for the obstacle problem for higher-order powers of the fractional Laplacian (see \cite{DHP2023} for some progress in this direction). If one assumes that the equilibrium measure has full support in $\R^\d$, or restricts to the torus where full support is easily established under a smallness condition on the confinement, then all terms involving $\zeta$ vanish. One can then treat the sub-Coulomb case by following the same proof in this paper, using instead the non-sharp commutator estimate of \cite{NRS2021}, leading to vanishing of $\Hr_{N}(\uz_N^t,u^t)$ as $\vep+\frac{1}{N}\rightarrow 0$, provided that
\begin{equation}
\lim_{\vep+\frac{1}{N}\rightarrow 0} \frac{1}{\vep^2}\pa*{N^{-\frac{2}{(\s+2)(\s+1)}}+\frac{\log N}{N}\indic_{\s=0}} =0,
\end{equation}
which we believe is a suboptimal scaling assumption.
\end{remark}

\subsection{Organization of article}\label{ssec:introOrg}
Let us comment on the organization of the remaining body of the paper.

In \cref{sec:EMOP}, we review the basic minimization problem for the equilibrium measure (\cref{ssec:EMOPem}) and the connection to the obstacle problem for the fractional Laplacian (\cref{ssec:EMOPop}). We also clarify the precise regularity and topological assumptions imposed on $\mu_V$ and its support in this paper (\cref{ssec:EMOPassmp}).

In \cref{sec:ME}, we review basic properties of the modulated potential energy in the form of its almost positivity (\cref{lem:MPElb}) and coercivity (\cref{lem:MPEcoer}). We also record the sharp commutator estimate from \cite{RS2022} referenced in the introduction (\cref{prop:comm}).

In \cref{sec:Lake}, we review the local well-posedness of the Lake equation \eqref{eq:Lake} in bounded domains subject to the no-flux boundary condition (\cref{prop:LakeWP}). We also review the procedure for extending the Lake equation solution to the whole space (\cref{lem:ugradVext}).

In \cref{sec:MP}, we prove our main result \cref{thm:mainSMF}. We start by giving a mathematically precise version of \cref{thm:mainSMF} in the form of \cref{thm:mainSMFrig} (\cref{ssec:MPmr}). {We then compute the differential identity (see \cref{lem:MEsmfID}) obeyed by the total modulated energy (\cref{ssec:MPcomp}). Finally, we show the Gr\"onwall relation for the total modulated energy, completing the proof of \cref{thm:mainSMF} (\cref{ssec:MPgron}).}

In \cref{sec:1DCou}, we present the explicit example in the one-dimensional case which shows that the empirical measure may have no weak limit if the scaling assumption \eqref{eq:SMFscl} fails. The main result is \cref{prop:1DCou}. 

There are two appendices. In \cref{sec:appreg}, we present the previously advertised generalization of our modulated-energy method to sufficiently regular interactions but which lack the Riesz-type structure that allows a commutator estimate of the type \cref{prop:comm} to hold. The main result of this appendix is \cref{thm:mainSMFrigreg}. In \cref{sec:appOP}, we review some facts about regularity theory for the fractional obstacle problem, as well as prove some new nondegeneracy results for the free boundary, which are needed for \cref{lem:ugradVext}.

\subsection{Acknowledgments}
The first author thanks the Institute for Computational and Experimental Research in Mathematics (ICERM) for its hospitality, where part of the research for this project was carried out during the Fall 2021 semester program ``Hamiltonian Methods in Dispersive and Wave Evolution Equations.'' He also thanks the Courant Institute of Mathematical Sciences at NYU for their hospitality during his visit in April 2024.
 Both authors thank Stephen Cameron and Xavier Ros-Oton for helpful remarks on the fractional obstacle problem.
 
\subsection{Notation}\label{ssec:preN}

We close the introduction with the basic notation used throughout the article without further comment, following the conventions of \cite{RS2022}.

Given nonnegative quantities $A$ and $B$, we write $A\lesssim B$ if there exists a constant $C>0$, independent of $A$ and $B$, such that $A\leq CB$. If $A \lesssim B$ and $B\lesssim A$, we write $A\sim B$. Throughout this paper, $C$ will be used to denote a generic constant which may change from line to line. Also $(\cdot)_+$ denotes the positive part of a number.

$\N$ denotes the natural numbers excluding zero, and $\N_0$ including zero. {For $N\in\N$, we abbreviate $[N]\coloneqq \{1,\ldots,N\}$.} $\R_+$ denotes the positive reals. Given $x\in\R^\d$ and $r>0$, $B(x,r), B_r(x)$ and $\p B(x,r), \p B_r(x)$ respectively denote the ball and sphere centered at $x$ of radius $r$. Given a distribution $f$, we denote its support by $\supp f$. The notation $\nabla^{\otimes k}f$ denotes the $k$-tensor field with components $(\p_{i_1}\cdots\p_{i_k}f)_{1\leq i_1,\ldots,i_k\leq \d}$.

$\P(\R^\d)$ denotes the space of Borel probability measures on $\R^\d$. If $\mu$ is absolutely continuous with respect to Lebesgue measure, we shall abuse notation by writing $\mu$ for both the measure and its density function. When the measure is clearly understood to be Lebesgue, we shall simply write $\int_{\R^{\d}}f$ instead of $\int_{\R^\d}fdx$.

{
$C^{k,\alpha}(\R^\d)$ denotes the inhomogeneous space of $k$-times differentiable functions on $\R^\d$ whose $k$-th derivative is $\alpha$-H\"{o}lder continuous, for $\al\in [0,1]$ (i.e. $\alpha=0$ is bounded and $\alpha=1$ is Lipschitz). As per convention, a $\dot{}$ superscript denotes the homogeneous space/seminorm.} {With a slight abuse of notation, we let $C^{\ga}$  denote the Besov space $B_{\infty,\infty}^{\ga}$, which coincides with the H\"older space $C^{k,\ga-k}$ when $k<\ga<k+1$ for integer $k$, but is strictly larger than the usual $C^\ga$ space for integer $\ga$.} {We let $H^{\ga}=W^{\ga,2}$ denote the standard $L^2$ Sobolev space of $f$ such that $(I-\Delta)^{\ga/2}f\in L^2$.} Finally, we let $\Sc$ and $\Sc'$ denotes the space of Schwartz functions and the space of tempered distributions, respectively.
\section{The equilibrium measure and the (fractional) obstacle problem}\label{sec:EMOP}

\subsection{The equilibrium measure}\label{ssec:EMOPem}
We need to clarify our assumptions on the external potential $V$ as it pertains to the minimizer $\mu_V$ for the energy $\Ec$ defined in \eqref{eq:Edef}. Throughout this section, we assume that $\g$ is of the form \eqref{eq:gmod}.

Our basic assumptions on $V$ to ensure the existence {and uniqueness} of the equilibrium measure $\mu_V$ are the following:
\begin{enumerate}[(i)]
\item\label{VeBA1} $V$ is lower semicontinuous (l.s.c.) and bounded below,
\item\label{VeBA2} $\{x\in\R^\d: V(x)<\infty\}$ has positive $\g$-capacity,
\item\label{VeBA3} {$\lim_{|x|\rightarrow\infty} V(x) +\g(x) = \infty$.}
\end{enumerate}

Under assumptions \ref{VeBA1}-\ref{VeBA3}, Frostman's theorem \cite{Frostman1935}  guarantees the existence of a minimizer to \eqref{eq:Edef} satisfying certain properties. {We refer to \cite[Chapter 2]{SerfatyLN} for a proof under these assumptions and for our class of $\g$'s. }

\begin{prop}\label{prop:Frost}
Assuming $V$ satisfies \ref{VeBA1}-\ref{VeBA3}, there exists a unique minimizer of $\Ec$ in $\P(\R^\d)$, denoted by $\mu_V$, with $\Ec(\mu_V)$ finite. Moreover, $\mu_V$ has the following properties:
\begin{itemize}
\item $\Sigma\coloneqq\supp \mu_V$ is bounded and has positive $\g$-capacity;
\item if we define the $\g$-potential $h^{\mu_V} \coloneqq \g\ast\mu_V$, then there exists a constant $c>0$ (called the Robin constant), such that
\begin{equation} \label{EulerLagrange}
\begin{cases}
h^{\mu_V} + V \geq c \quad \text{quasi-everywhere (q.e.)},\footnotemark \\
h^{\mu_V} + V = c \quad \text{q.e. on} \ \Sigma.
\end{cases}
\end{equation}
\footnotetext{Quasi-everywhere means the exceptional set has zero $\g$-capacity, which is stronger than zero Lebesgue measure. See \cite[Section II.1]{Landkof1972}.}
\end{itemize}
\end{prop}
{Remark that the second item uniquely characterizes the equilibrium measure}.

If we define the function
\begin{equation}\label{eq:zetadef}
\zeta \coloneqq h^{\mu_V} +V-c,
\end{equation}
then \cref{prop:Frost} implies that $\zeta\geq 0$ on $\R^\d$. For our purposes, we need to assume that $\mu_V$ is absolutely continuous with respect to the Lebesgue measure and has a sufficiently regular density (at least {bounded} in $\Sigma$), which we also denote by $\mu_V$ with an abuse of notation. {We also need to assume that $\Sigma$ equals the closure of its interior, i.e. $\Sigma=\ol{\overset\circ\Sigma}$, and that the boundary $\p\Sigma$ is sufficiently smooth, at least $C^{1,1}$.}

Let us comment on examples of possible external potentials $V$. In the case where $V(x)=\frac12|x|^2$ and $\g$ is the Coulomb potential, it is well-known that $\mu_V$ is a constant multiple of the characteristic function of a ball centered at the origin. More generally, if $V$ is still quadratic and $\g$ is of the form \eqref{eq:gmod}, then $\mu_V$ is the so-called fractional Barenblatt profile \cite[Theorem 2.2]{BIK2015} and \cite[Theorem 3.1]{CV2011}  which generalizes the classical Barenblatt profile for self-similar solutions of the porous medium equation. 
 Let us also mention that if $V$ is radial, then $\Sigma$ is always a ball. Furthermore, given any probability measure $\mu_*$ with finite $\g$-energy, one can choose the confinement
\begin{align}
V \coloneqq -\g\ast\mu_* + \frac12\int_{(\R^\d)^2}\g(x-y)d\mu_*(x)d\mu_*(y),
\end{align}
so that the associated equilibrium measure $\mu_V = \mu_*$. 
{Remark that for this choice of $V$, the energy \eqref{eq:Edef} becomes the \emph{maximum mean discrepancy (MMD)} widely used in statistics and machine learning, e.g. \cite{GBRSS2012}.}

\subsection{The equilibrium measure as a solution to the obstacle problem}\label{ssec:EMOPop}

As {emphasized} in \cite[Section 2.5]{Serfaty2015CGL}  (see also \cite{PS2017,CDM2016} for the fractional case), the minimization problem for $\Ec$ is intimately connected to an obstacle problem for the fractional Laplacian: for $s\in (0,1]$, given an obstacle $\varphi:\R^\d\rightarrow\R$, find a function $w:\R^\d\rightarrow\R$ such that
\begin{equation}\label{eq:OPfl}
\min\{(-\D)^{s}w, w-\varphi\} = 0,
\end{equation}
which, for instance, is reviewed in \cite{Caffarelli1998} for the classical case $s=1$ and has been studied in \cite{Silvestre2007,ACS2008, CSS2008, KPS2015, DsS2016, BFRo2018, JN2017, KRS2019} for the fractional case $s \in (0,1)$. The set $\{w=\varphi\}$ is called the \emph{coincidence set} or \emph{contact set}, which is an unknown, and its boundary $\p\{w=\varphi\}$ is called the \emph{free boundary}. More precisely, if $\mu_V$ minimizes the energy $\Ec$, then it follows from \cref{prop:Frost} that $h^{\mu_V}$ satisfies \eqref{eq:OPfl} with $s = \frac{\d-\s}{2}$ and obstacle $\varphi = c-V$. Moreover, $\{\zeta=0\}$ is the contact set. 

{We have $\Sigma \subset \{\zeta=0\}$}, but in general it is not true that $\{\zeta=0\} =\Sigma$,  the latter of which is called the \emph{droplet} \cite{HM2013}. To avoid this possible issue, we hypothesize that $\{\zeta=0\}=\Sigma$. A sufficient condition to ensure this equality is that  $\D\varphi <0$, equivalently $\D V>0$, in a neighborhood of $\{\zeta=0\}$. 

The connection between the minimization problem and the obstacle problem allows us to access regularity results for the latter, in particular those pertaining to the regularity of the free boundary {and the so-called ``lift-off" rate from the obstacle, i.e.~growth rate of $\zeta$ away from $\Sigma$. The crucial point for us is that  lift-off be sufficiently fast, which is the main reason for the assumption that $\Sigma= \{\zeta=0\}$ as well as  our other assumptions.}  

We now briefly review a few important facts, focusing on the fractional case $s \in (0,1)$. Additional discussion and results are deferred to \cref{sec:appOP}. 

As shown in \cite{CSS2008}, improving upon \cite{Silvestre2007}, the optimal regularity for the solution $w$ of \eqref{eq:OPfl} is $C^{1,s}$, under the assumption that $\varphi$ is sufficiently regular (e.g.~$C^{3,\al}$). Consequently, $\nabla\zeta(x) = 0$ for all $x\in\Sigma$.  \cite{CSS2008} additionally classifies free boundary points $x_0$ as regular points, if the blow-up at $x_0$ has $1+s$ homogeneity,\footnote{{Strictly speaking, \cite{CSS2008} did not show the convexity of blow-ups, which is an important detail that was later addressed in \cite{fRJ2021, CDV2022}.}} or singular points, if the contact set has zero density at $x_0$; and establishes $C^{1,\al}$ regularity at regular points. In principle, the union of regular points and singular points does not exhaust the free boundary; but \cite{BFRo2018} shows that such points are the only two possibilities, under the assumption that $\D\varphi\leq -c< 0$. \cite{JN2017, KRS2019} independently establish higher regularity of the free boundary near regular points {(see also \cite{ArO2020} for improvements and generalizations)}. Paraphrasing from the former work, if the obstacle $\varphi\in C^{m,\be}(\R^\d)$ for $m\geq 4$ and $\be\in (0,1)$ and $x_0\in \p\Sigma$ is a regular point of the free boundary, then $\p\Sigma$ is $C^{m-1,\al}$ in a neighborhood of $x_0$ for some $\al\in (0,1)$ depending $\d,s,m,\be$; in particular, if $\varphi\in C^\infty$, then the free boundary is $C^\infty$. In general, the free boundary may contain both regular and singular points, around the latter of which the free boundary is only $C^1$. We always assume in this work that this is not the case. The recent \cite{FRoS2020} proves for the classical obstacle problem that the absence of singular points is generic for $\d\leq 4$. 

\begin{remark}
For many results, the obstacle $\varphi$ and its derivatives up to certain order need to be bounded. This is obviously not the case if $\varphi = V-c$ for $V$ as above. However, our assumption that $\Sigma$ is bounded allows us to work with a local obstacle problem in a neighborhood of $\Sigma$. In which case, the needed boundedness for $\varphi$ and its derivatives is satisfied.
\end{remark}

The main application of this regularity theory for the obstacle problem comes in the form of the following lemma, which is a Riesz generalization of the Coulomb-specific result \cite[Lemma 3.2]{BCGM2015} (see also Lemma 3.1 in the cited work for the specific case of a quadratic potential). As discussed in the introduction, we shall need this lemma to estimate one of the terms appearing in the evolution equation satisfied by the total modulated energy $\Hr_N(Z_N^t,u^t)$. {The proof is deferred to the end of \cref{sec:appOP}.} 

\begin{lemma}\label{lem:ugradVext}
Let $\theta>2$ and suppose that $V\in C_{loc}^{\theta+\frac{\d-\s}{2}}$. There exists a constant $C>0$ depending on $\d,\s,V,\Sigma$ such that for any vector field $v:\R^\d\rightarrow\R^\d$  {satisfying the no-flux condition $v\cdot \nu =0 $ on $\partial \Sigma$ (where $\nu$ is the unit normal to $\partial \Sigma$)}, with $\supp v$ contained in a $2\diam(\Sigma)$-neighorhood of $\Sigma$, it holds that
\begin{equation}\label{eq:ugradzeta}
\left|v(x)\cdot\nabla\zeta(x)\right| \leq C\|v\|_{W^{1,\infty}}\zeta(x), \qquad \forall x\in\R^\d.
\end{equation}
\end{lemma}

\subsection{Assumptions}\label{ssec:EMOPassmp}
 For convenience of subsequent referencing, we explicitly record below the assumptions made on $V,\mu_V$, and $\Sigma$ in the preceding paragraphs:
\begin{enumerate}[(H1)]
\item\label{assSigeq} $\Sigma$ coincides with the coincidence set and equals the closure of its interior, i.e. $\Sigma=\ol{\overset\circ\Sigma}=\{\zeta=0\}$;
\item\label{assSigcon} $\Sigma$ is connected;
\item\label{assSigrp} every $x_0\in \p\Sigma$ is a regular point;
\item\label{assVreg} $V$ is locally $C^{2+\frac{\d-\s}{2}+\epsilon}$ for $\epsilon>0$;
\item\label{assVLapl} $\Delta V>0$ in a neighborhood of $\{\zeta =0\}$.
\end{enumerate}
Further regularity assumptions on the density $\mu_V$ and its support $\Sigma$ will be made in \Cref{sec:Lake,sec:MP}.

\section{The modulated potential energy and commutators}\label{sec:ME}

We review basic properties of the modulated potential energy
\begin{align}\label{eq:MEdef}
\Fr_N(\ux_N,\mu) \coloneqq \frac12\int_{(\R^\d)^2 \setminus\triangle}\g(x-y)d\Big(\frac1N\sum_{i=1}^N\delta_{x_i}-\mu\Big)^{\otimes2}(x,y)
\end{align}
for (super-)Coulomb Riesz interactions \eqref{eq:gmod} and the sharp estimate for their first variation along a transport, which will be used in this paper as advertised in the introduction. In the remainder of this section, we assume that $\mu$ is distribution in $L^1(\R^\d) \cap L^\infty(\R^\d)$ such that $\int d\mu = 1$.\footnote{The $L^\infty$ assumption may be relaxed to an $L^p$ assumption for $p=p(\d,\s)$ at the cost of larger additive errors in the estimates. For instance, see \cite[Section 3]{Rosenzweig2022, Rosenzweig2022a}.} {If $\s\le 0$, we suppose further that $\int_{(\R^\d)^2}|\g(x-y)|d|\mu|^{\otimes 2}(x,y)<\infty$.} These conditions are sufficient to ensure that $\Fr_N(\XN,\mu)$ is well-defined. 
We emphasize that we do not require that $\mu\ge 0$, only that $\int d\mu =1$. This is an important {generalization}, originating in \cite{Rosenzweig2021ne}, because the distribution $\mu^t+\vep^2\Uu^t$ is not necessarily nonnegative, though always has mass one since $\Uu^t$ has zero mean.

The quantity \eqref{eq:MEdef} first appeared as a \emph{next-order electric energy} in \cite{SS2015,RS2016,PS2017} and was subsequently used in the dynamics context in \cite{Duerinckx2016,Serfaty2020} and following works---in the spirit of Brenier's total modulated energy \cite{Brenier2000}. In that context, the term ``modulated energy'' was used, instead of ``modulated potential energy,'' as there is no  modulated kinetic energy necessitating distinction. Concretely, {$\Fr_N$} is the total interaction of the system of $N$ discrete charges located at $\XN$ against a negative (neutralizing) background charge density $\mu$, with the infinite self-interaction of the points removed.

As shown in the aforementioned prior works, $\Fr_N$ is not necessarily positive; however, it effectively acts as a squared distance between the spatial empirical measure $\frac1N\sum_{i=1}^N \delta_{x_i}$ and $\mu$ and is bounded from below, as expressed by the next two lemmas. The logarithmic correction in the $\s=0$ case is present because $\Fr_N(\ux_N,\mu)$ is not quite the right quantity to consider, since it is not invariant under zooming into the microscale $(N\|\mu\|_{L^\infty})^{-1/\d}$. For a proof of \cref{lem:MPElb}, we refer to  \cite[Proposition 2.3]{RS2022}, which improves upon \cite[Corollary 3.4]{Serfaty2020}. For a proof of \cref{lem:MPEcoer}, we refer to \cite[Lemma 3.1, Corollary 3.3]{SerfatyLN}, which improves upon \cite[Proposition 3.6]{Serfaty2020}.

\begin{lemma}\label{lem:MPElb}
There exists a constant $C>0$, depending only on $\d,\s$, such that for any pairwise distinct configuration $\XN\in(\R^\d)^N$,
\begin{align}\label{eq:MPElb}
\Fr_N(\ux_N,\mu) + \frac{\log(N\|\mu\|_{L^\infty})}{2\d N}\indic_{\s=0} \ge -C\|\mu\|_{L^\infty}^{\frac{\s}{\d}}N^{\frac{\s}{\d}-1}.
\end{align}
\end{lemma}

\begin{lemma}\label{lem:MPEcoer}
There exists a constant $C>0$, depending only on $\d,\s$, such that for any test function $\phi$ and pairwise distinct configuration $\XN\in(\R^\d)^N$,
\begin{multline}
\Big|\int_{\R^\d}\phi \, d\Big(\frac1N\sum_{i=1}^N \delta_{x_i} - \mu\Big)\Big| \leq  C\|(-\Delta)^{\frac{\d-\s}{2}}\phi\|_{L^\infty} (N\|\mu\|_{L^\infty})^{\frac{\s}{\d}-1} \\
+ C\|\phi\|_{\dot{H}^{\frac{\d-\s}{2}}}\Big(\Fr_N(\ux_N,\mu) + \frac{\log(N\|\mu\|_{L^\infty})}{2\d N}\indic_{\s=0} + C\|\mu\|_{L^\infty}^{\frac\s\d}N^{\frac{\s}{\d}-1}\Big)^{1/2}.
\end{multline}
Consequently, if $\ka>\d-\s+\frac{\d}{2}$, then
\begin{align}
\|\frac1N\sum_{i=1}^N\delta_{x_i}-\mu\|_{H^{-\ka}} \le C(N\|\mu\|_{L^\infty})^{\frac{\s}{\d}-1}  + C \Big(\Fr_N(\ux_N,\mu) + \frac{\log(N\|\mu\|_{L^\infty})}{2\d N}\indic_{\s=0} + C\|\mu\|_{L^\infty}^{\frac\s\d} N^{\frac{\s}{\d}-1}\Big)^{1/2}.
\end{align}
\end{lemma}

The next proposition asserts a functional inequality that controls the first variation of the modulated potential energy along the transport map $\I+tv$ in terms of the modulated potential energy itself. We recall from the introduction that such a functional inequality is equivalent to a commutator estimate. The specific estimate below is taken from the authors' recent \cite[Theorem 1.1]{RS2022}, which improves upon earlier nonsharp estimates \cite{Duerinckx2016,Serfaty2020} and generalizes to the super-Coulomb Riesz case sharp Coulomb-specific estimates \cite{LS2018, Serfaty2023, Rosenzweig2021ne}. As remarked in the introduction, the sharpness of the $N^{\frac{\s}{\d}-1}$ additive error is crucial to allowing for the scaling relation \eqref{eq:SMFscl} between $\vep$ and $N$. 


\begin{prop}\label{prop:comm}
There exists a constant $C>0$ depending only on $\d,\s$ such that for any Lipschitz vector field $v:\R^\d\rightarrow\R^\d$ and any pairwise distinct configuration $\ux_N \in (\R^\d)^N$, it holds that\footnote{Here and throughout this paper, $\|\nab v\|_{L^\infty}$ denotes $\| |\nab v|_{\ell^2}\|_{L^\infty}$.}
\begin{multline}\label{main1}
\Big|\int_{(\R^\d)^2\setminus\triangle}(v(x)-v(y))\cdot\nabla\g(x-y)d\Big(\frac{1}{N}\sum_{i=1}^N\delta_{x_i} - \mu\Big)^{\otimes 2}(x,y)\Big| \\
\leq C\|\nabla v\|_{L^\infty}\Big( \Fr_N(\ux_N,\mu) - \frac{\log(N\|\mu\|_{L^\infty})}{2\d N}\indic_{\s=0} + C{\|\mu\|_{L^\infty}^{\frac\s\d}N^{\frac\s\d -1}} \Big).
\end{multline}
\end{prop}

\section{The Euler/Lake equation }\label{sec:Lake}


In this section, we review some properties of solutions to the Lake equation \eqref{eq:Lake}. Well-posedness has been studied, for instance, in \cite{LOT1996phy,LOT1996,Oliver1997,LO1997, Huang2003, BM2006, LNP2014,Duerinckx2018,BJ2018,HLM2022, aTL2023}. For our purposes, we need existence and uniqueness of classical solutions in a bounded domain (corresponding to the interior $\overset\circ\Sigma$ of $\Sigma=\supp\mu_V$)  under minimal topological assumptions.  Namely, the Cauchy problem is
\begin{align}\label{eq:LakeCP}
\begin{cases}
\p_t u + \ga u + (u\cdot\nab) u  = -\nab p & \quad \text{in} \ \overset\circ\Sigma\\
\div(\mu_V u) = 0 & \quad \text{in} \ \overset\circ\Sigma\\
u \cdot \nu = 0 & \quad \text{on} \ \p\overset\circ\Sigma \\
u^0 = u^\circ & \quad \text{in} \ \overset\circ\Sigma
\end{cases}
\end{align}
Working in a bounded domain makes physical sense because taking the quasineutral limit means that the density of particles outside of $\Sigma$ is zero and  there is no current flux across the boundary of $\Sigma$.

{For the purposes of our main result \cref{thm:mainSMF} (more precisely, its rigorous version \cref{thm:mainSMFrig}), we need at least a solution $u\in L^\infty([0,T], W^{1,\infty}(\overset\circ\Sigma) \cap H^\sigma(\overset\circ\Sigma))$ of \eqref{eq:LakeCP} for large enough $\sigma\ge 0$. The existence and uniqueness of such a solution under the strong assumption $\inf_{\Sigma} \mu_V >0$ is shown in \cite[Theorem A.1]{BCGM2015}, which we quote below in \cref{prop:LakeWP}.} 

\begin{prop}\label{prop:LakeWP}
Let $\Omega\subset\R^\d$ be a bounded open set with $\p\Om$ of class $C^{\sigma}$ for integer $\sigma>\frac{\d+2}{2}$. Suppose that $\mu_V: \Om\rightarrow\R$ is in $H^{\sigma+1}$, such that $\inf_{\Om}\mu_V > 0$.\footnote{Note that our assumption $\sigma>\frac{\d+2}{2}$ \emph{a fortiori} implies that $\mu_V$ is H\"older continuous by Sobolev embedding. In particular, $\mu_V$ makes sense pointwise.} Given a vector field $u^\circ:\Om\rightarrow\R^\d$ such that $\div(\mu_V u^\circ) =0$ in $\Omega$ and satisfying the no-flux condition $u^\circ\cdot\nu=0$ on $\p\Om$, where $\nu$ is the unit outward normal field for $\p\Om$, and a friction coefficient $\ga\in\R$, there exists a $T>0$ and a solution $u\in L^\infty([0,T]; H^\sigma(\Om))$ of \eqref{eq:Lake} satisfying the no-flux condition for $t\in[0,T]$. Moreover,
\begin{equation}
\sup_{0\leq t\leq T} \pa*{\|u^t\|_{H^\sigma} + \|\p_t u^t\|_{H^{\sigma-1}} + \|\nabla p^t\|_{H^\sigma} + \|\p_t\nabla p^t\|_{H^{\sigma-1}}} \leq C<\infty,
\end{equation}
where $C$ depends only on $\d,\Sigma,\ga,T,\|u^\circ\|_{H^\sigma}$.
\end{prop}

{
This assumption $\inf_{\Sigma} \mu_V >0$ is overly restrictive in the non-Coulomb case $\s\neq\d-2$. For example, the equilibrium measure for $V=|x|^2$ vanishes continuously on the boundary $\p\Sigma$, unlike the Coulomb case. Given $\mu_V$, we therefore take as an \emph{assumption} the existence of a solution $u\in L^\infty([0,T], W^{1,\infty}(\overset\circ\Sigma) \cap H^\sigma(\overset\circ\Sigma))$. Proving the well-posedness of the Cauchy problem without the assumption $\inf_{\Sigma} \mu_V >0$ is beyond the scope of this paper, and we refer, for instance, to \cite{BM2006, aTL2023} for some results in the $\d=2$ case. 

Since the particles at the microscopic level of \eqref{eq:NewODE} move through the whole space and are not confined to the support of the equilibrium measure, it is convenient to regard the solution $u$ to the Lake equation given by \cref{prop:LakeWP} as a vector field in $\R^\d$. A naive extension $\tl u \coloneqq u\indic_{\Omega}$, however, does not preserve the regularity of $u$ nor the divergence-free condition. As partially sketched in \cite[Remark B.2]{BCGM2015} (see \cite[Lemma B.1]{BCGM2015} for when $\Omega$ is explicitly a ball, or more generally, an ellipsoid), it is possible to find a compactly supported extension $\tl u$ with the same regularity as the solution $u$ and satisfying $\div(\mu_V\tl u)=0$ {in $\R^\d$ in the sense of distributions.} We present the details for this construction.

Let $\Omega$ be a bounded, Lipschitz domain and $v\in W^{m,\infty}(\Omega) \cap H^\sigma(\Omega)$ for some $m\ge 0$ and $\sigma\ge 0$. By Stein's extension theorem \cite[Theorem 5, VI.3]{Stein1970},\footnote{Strictly speaking, Stein's theorem only implies that there exists a bounded linear extension operator $\mathfrak{G}: W^{\sigma,p}(\Om)\rightarrow W^{\sigma,p}(\R^\d)$ for integer $\sigma\geq 0$. But the boundedness of the operator on fractional Sobolev spaces follows from the theory of complex interpolation (e.g. see \cite[Theorem 4.1.2, Theorem 6.4.5]{BL1976}).} there exists a vector field $\tl{v} \in W^{m,\infty}(\R^\d) \cap H^\sigma(\R^\d)$ such that
\begin{align}
\|\tl{v}\|_{W^{m,\infty}(\R^\d)} &\le C \|v\|_{W^{m,\infty}(\Omega)}, \\
\|\tl{v}\|_{H^\sigma(\R^\d)} &\le C\|v\|_{H^\sigma(\Omega)},
\end{align}
where $C>0$ depends only $\d,m,\sigma,\Omega$. Let $\mathcal{O}$ be an open neighborhood of $\bar\Omega$. Then we can we find a $C^\infty$ bump function $\chi$ such that $\chi\equiv 1$ on $\bar\Omega$ and $\supp\chi\subset \mathcal{O}$. Multiplying $\tl{v}$ by $\chi$, we may assume without loss of generality that $\tl{v}$ is supported in $\mathcal{O}$.

We apply the preceding result with $\Omega = \overset\circ\Sigma$. If $\div(\mu_V v) = 0$ in $\Omega$ and $v$ satisfies the no-flux condition, then we claim that $\div(\mu_V\tl{v}) = 0$ in $\R^\d$ in the distributional sense. Indeed, given any test function $\varphi$, it follows from the definition of the distributional derivative and integration by parts that
\begin{align}
-\int_{\R^\d}\varphi \div(\mu_V\tl{v}) = \int_{\R^\d}\nab \varphi\cdot \tl{v}\mu_V  = \int_{\overset\circ\Sigma}\nab\varphi\cdot v\mu_V = \int_{\p\overset\circ\Sigma}\varphi (v\cdot\nu)\mu_V - \int_{\overset\circ\Sigma}\varphi\div(\mu_V v) = 0.
\end{align}

Given a solution $u\in L^\infty([0,T], W^{m,\infty}(\overset\circ\Sigma)\cap H^\sigma(\overset\circ\Sigma))$ to \eqref{eq:LakeCP}, we can apply the preceding lemma pointwise in $t$ with $v=u^t$ to obtain the following proposition.

\begin{prop}\label{lem:VFext}
Suppose that $\overset\circ\Sigma$ is a Lipschitz domain and $u \in L^\infty([0,T], W^{m,\infty}(\overset\circ\Sigma)\cap H^\sigma(\overset\circ\Sigma))$ for $m,\sigma\ge 0$. Given any open neighborhood of $\mathcal{O}\supset \Sigma$, there exists $\tl{u}\in L^\infty([0,T], W^{m,\infty}(\R^\d) \cap H^\sigma(\R^\d))$ such that for every $t\ge 0$, $\tl{u}^t$ is compactly supported in $\mathcal{O}$, $\tl{u}^t = u^t$ in $\overset\circ\Sigma$, $\tl{u}^t\cdot \nu = 0$ on $\p\overset\circ\Sigma$, and $\div(\mu_V\tl{u}^t) = 0$ in $\R^\d$. Moreover,
\begin{align}
\|\tl{u}\|_{L^\infty([0,T], W^{m,\infty}(\R^\d))} &\le C \|u\|_{L^\infty([0,T], W^{m,\infty}(\overset\circ\Sigma))}, \\
\|\tl{u}\|_{L^\infty([0,T], H^\sigma(\R^\d))} &\le C \|u\|_{L^\infty([0,T], H^\sigma(\overset\circ\Sigma))},
\end{align}
where $C>0$ depends only on $\d,m,\sigma,\overset\circ\Sigma$. 
\end{prop}
}

\begin{remark}\label{rem:ptuext}
We can deduce the regularity of the time derivatives of the extension, that is $\p_t^k\tl u$ for $k\ge 1$. Indeed, suppose that $\p_t^k u \in L^\infty([0,T], W^{m-k,\infty}(\overset\circ\Sigma)\cap  H^{\sigma -k}(\overset\circ\Sigma))$. 
Now since the difference quotients $\frac{\p_t^{k-1}u(t+h) - \p_t^{k-1}u(t)}{h}$ converge, as $h\rightarrow 0$, to $\p_t^k u$ in $H^{\sigma-k}(\overset\circ\Sigma)$ (resp. $W^{m-k,\infty}(\overset\circ\Sigma)$), it follows from the continuity of the linear map $u\mapsto \tl u$ that
\begin{equation}
\p_t\widetilde{\p_t^{k-1} u}(t)= \lim_{h\rightarrow 0} \frac{\widetilde{\p_t^{k-1} u}(t+h) - \widetilde{\p_t^{k-1} u}(t)}{h} \quad \text{exists in} \quad H^{\sigma-k}(\R^\d) \text{ and equals $\widetilde{\p_t^k u}(t)$},
\end{equation}
where $\widetilde{f}$ denotes the extension operator constructed above applied to $f$. By induction on $k$, it follows that $\p_t^k\tl{u} = \widetilde{\p_t^k u}$.  Moreover,
\begin{align}
\|\p_t^k\tl u^t\|_{H^{\sigma-k}(\R^\d)} &\leq C\|\p_t^k u^t\|_{H^{\sigma-k}(\overset\circ\Sigma)},\\
\|\p_t^k \tl{u}^t\|_{W^{m-k,\infty}(\R^\d)} &\le C\|\p_t^k u^t\|_{W^{m-k,\infty}(\overset\circ\Sigma)},
\end{align}
where the constant $C>0$ is independent of $t$. Thus, $\p_t^k\tl u\in L^\infty([0,T], W^{m-k,\infty}(\R^\d)\cap H^{\sigma-k}(\R^\d))$.
This will be of use in \cref{sec:MP} when we incorporate the corrector $\Uu$, which is constructed from $\tl u, \p_t\tl u$, into our modulated energy. Note that $\tl u$ \emph{does not} satisfy equation \eqref{eq:Lake} in all of $\R^\d$, only in the open set $\overset\circ\Sigma$. Therefore, we cannot deduce that $\p_t\tl u\in L^\infty([0,T], W^{m-1,\infty}(\R^\d) \cap H^{\sigma-1}(\R^\d))$ simply using the equation \eqref{eq:Lake}.
\end{remark}

\section{Main proof}\label{sec:MP}
In this section, we prove our main result \cref{thm:mainSMF}.

\subsection{Statement of main result}\label{ssec:MPmr}

We present a rigorous, precise statement of our main result, previously stated informally in \cref{thm:mainSMF}.

For $m\ge 0$ and $\sigma\ge 0$, given a solution $u \in L^\infty([0,T], W^{m,\infty}(\overset\circ\Sigma) \cap H^\sigma(\overset\circ\Sigma))$ to the Lake equation \eqref{eq:Lake}, such that $\p_t^k u \in L^\infty([0,T], W^{m-k,\infty}(\overset\circ\Sigma)\cap H^{\sigma-k}(\overset\circ\Sigma))$, let $\tl u:\R^\d\rightarrow\R^\d$ be its extension in $L^\infty([0,T], W^{m,\infty}(\R^\d) \cap H^\sigma(\R^\d))$ such that $\p_t^k\tl{u} \in L^\infty([0,T], W^{m-k,\infty}(\R^\d)\cap H^{\sigma-k}(\R^\d))$, $\tl{u}^t\cdot\nu$ on $\p\Sigma^\circ$, and $\div(\mu_V\tl{u})=0$, as constructed in \cref{sec:Lake}.  Hereafter to, we work exclusively with the extension and therefore drop the $\tl{\cdot}$ superscript. Given a solution $\uz_N^t = (\ux_N^t,V_N^t)$ to the $N$-particle system \eqref{eq:NewODE}, we recall from \cref{ssec:introMR} the {total modulated energy}
\begin{equation}\label{eq:MEcon}
\Hr_{N}(\uz_N^t,u^t) \coloneqq \frac{1}{2N}\sum_{i=1}^N |v_i^t-u^t(x_i^t)|^2 + \frac{1}{\vep^2}\Fr_N(\ux_N^t,\mu_V+\vep^2\Uu^t)  + \frac{1}{\vep^2 N}\sum_{i=1}^N \zeta(x_i^t),
\end{equation}
where $\zeta$ is as defined in \eqref{eq:zetadef} and $\Uu:\R^\d\rightarrow\R$ is the corrector defined by
\begin{equation}\label{eq:Uudef}
\Uu \coloneqq (-\D)^{\frac{\d-2-\s}{2}}\div\pa*{\p_t u+\ga u + u\cdot\nabla u}.
\end{equation}
The motivation for the inclusion of this corrector will become clear during the computation behind the proof of \cref{lem:MEsmfID}.

\begin{remark}
Let us comment on the regularity of the corrector $\Uu$. Note that $-2 <\d-2-\s \le 0$ by assumption, so that $(-\D)^{\frac{\d-2-\s}{2}}\div$ is order $\d-1-\s$,  which acts like a fractional derivative  if $\d-2\leq \s<\d-1$, is of order zero if $\s=\d-1$, and is smoothing if $\d-1<\s<\d$. By \cref{lem:VFext} and \cref{rem:ptuext}, $(\p_t u+\ga u + u\cdot\nabla u)\in L^\infty([0,T], W^{m-1,\infty}(\R^\d) \cap H^{\sigma-1}(\R^\d))$, hence {$\Uu \in L^\infty([0,T], C^{m+\s-\d-\epsilon}(\R^\d)\cap  H^{\sigma+\s-\d}(\R^\d))$ for any $\epsilon>0$.\footnote{There is an $\epsilon$-loss of regularity because $\|(-\Delta)^{\frac{\ga}{2}}u\|_{L^\infty} < \infty$ does not imply $u\in \dot{C}^{\ga}$.} Appealing to \cref{rem:ptuext} again, $\p_t\Uu\in L^\infty([0,T], C^{m+\s-\d-1-\epsilon}(\R^\d)\cap H^{\sigma+\s-\d-1}(\R^\d))$.}
\end{remark}

 The precise version of \cref{thm:mainSMF} is the following result. {As explained in \cref{ssec:introMPf}, vanishing of the total modulated energy implies the weak-* convergence of the empirical measure to $\mu_V(x)\delta_{u^t(x)}$ as $N+\vep^{-1}\rightarrow \infty$.} 
 
\begin{thm}\label{thm:mainSMFrig}
{Suppose that $m>1+\d-\s$ and $\sigma\ge 2+\frac{\d-\s}{2}$}, and let $u$ be as above. Suppose that $\mu_V,\Sigma,V$ satisfy the assumptions \ref{VeBA1} - \ref{VeBA3} and \ref{assSigeq} - \ref{assVLapl} as listed in \cref{ssec:EMOPassmp}. For $0<\vep^2<\frac{\|\mu_V\|_{L^\infty}}{2\|\Uu\|_{L^\infty([0,T],L^\infty)}}$ and for $C>0$ sufficiently large, define the quantity
\begin{equation}\label{eq:Hscrdef}
\mathscr{H}_{N}^t \coloneqq \Hr_{N}(\uz_N^t,u^t) + \frac{\log(\|\mu_V+\vep^2\Uu^t\|_{L^\infty}N)}{2\vep^2 N\d}\indic_{\s=0} + \frac{C\|\mu_V+\vep^2\Uu^t\|_{L^\infty}^{\frac{\s}{\d}}N^{\frac{\s}{\d}-1}}{\vep^2}, 
\end{equation}
which is nonnegative,
\footnote{ Evidently, the modulated kinetic energy in the definition of \eqref{eq:MEcon} is nonnegative. The modulated potential energy is, in general, not nonnegative; however, the addition of the last two terms in \eqref{eq:Hscrdef} ensure nonnegativity of $\Fr_N(\ux_N^t,\mu_V +\vep^2\Uu^t)$ by \cref{lem:MPEcoer}.}
for a solution $Z_N^t$ of \eqref{eq:NewODE} and an admissible extension $u^t$ of a solution of \eqref{eq:Lake}.  Then there exist constants $C_1,C_2>0$ depending only on $\d,\s,\Sigma$, such that for every $t\in [0,T]$,
\begin{multline}\label{eq:SMFrig}
\mathscr{H}_{N}^t \leq e^{C_1\int_0^t \pa*{1+\|\nabla u^\tau\|_{L^\infty} + \|u^\tau\|_{W^{1,\infty}} + (\|\nabla u^\tau\|_{L^\infty}-\ga)_+}d\tau}\Bigg(\mathscr{H}_{N}^0 \\
+ \Big(\sup_{0\le \tau\le t} \frac{\log(\frac{\|\mu_V+\vep^2\Uu^\tau\|_{L^\infty}}{\|\mu_V+\vep^2\Uu^0\|_{L^\infty}})}{2\d N\vep^2}\indic_{\s=0}   + \frac{C(\|\mu_V+\vep^2\Uu^\tau\|_{L^\infty}^{\frac{\s}{\d}}- \|\mu_V+\vep^2\Uu^0\|_{L^\infty}^{\frac\s\d})N^{\frac\s\d - 1}}{\vep^2}\Big)\Big)_{+}\\
 +C_2\vep^2\int_0^t \pa*{\|\p_t\Uu^\tau\|_{\dot{H}^{\frac{\s+2-\d}{2}}} + \|u^\tau\|_{\dot{H}^{\frac{\s+4-\d}{2}}}\|\Uu^\tau\|_{L^\infty} + \|u^\tau\|_{L^\infty}\|\Uu^\tau\|_{\dot{H}^{\frac{\s+4-\d}{2}}} }^2 d\tau\\
+ C_2N^{\frac{\s}{\d}-1}\int_0^t \pa*{\|\p_t\Uu^\tau\|_{L^\infty}  + \|\div u\|_{L^\infty}\|\Uu\|_{L^\infty} + \|u\|_{L^\infty}\|\nab \Uu\|_{L^\infty}} d\tau \Bigg).
\end{multline}
\end{thm}

\begin{remark}
The reader may check that our regularity assumptions for $u,\mu_V$ imply that all the norms above are finite.
\end{remark}

\begin{remark}
At the risk of being too elementary, let us explain why \eqref{eq:SMFscl} implies that the right-hand side of \eqref{eq:SMFrig} vanishes as $\vep+N^{-1}\rightarrow 0$. It is clear that the last two lines vanish, as well as the second term on the second line. For the first term on the second line, we have by the triangle inequality and our smallness assumption for $\vep$ that
\begin{equation}
\frac{1}{3}\leq \frac{\|\mu_V+\vep^2\Uu^t\|_{L^\infty}}{\|\mu_V+\vep^2\Uu^0\|_{L^\infty}} \leq 3 \Longrightarrow \frac{\left|\log\pa*{\frac{\|\mu_V+\vep^2\Uu^t\|_{L^\infty}}{\|\mu_V+\vep^2\Uu^0\|_{L^\infty}}}\right|}{2\d N\vep^2} \leq \frac{\log 3}{2\d N\vep^2},
\end{equation}
which vanishes as $\vep+N^{-1}\rightarrow 0$ if \eqref{eq:SMFscl} holds. By similar reasoning, $\mathscr{H}_{N}^0$ vanishes as $\vep+N^{-1}\rightarrow 0$. 
\end{remark}

\subsection{Computation of the time-derivative of the total modulated energy}\label{ssec:MPcomp}
We turn to the proof of \cref{thm:mainSMFrig}. In this subsection, we compute the time derivative of $\Hr_{N}(\uz_N^t,u^t)$, the end result of our efforts stated in the following lemma (cf. \cite[Equation (2.7)]{HkI2021}).

\begin{lemma}\label{lem:MEsmfID}
If $\Uu^t  \coloneqq \cd(-\D)^{\frac{\d-2-\s}{2}}\div(\p_t u + \gamma u + u\cdot\nabla u)$, then
\begin{multline}\label{eq:dtHNid}
\frac{d}{dt}\Hr_{N}(\uz_N^t,u^t) = \frac{1}{N}\sum_{i=1}^N \pa*{v_i^t - u^t(x_i^t)}\cdot \pa*{(u^t(x_i^t)- v_i^t)\cdot\nabla u^t}(x_i^t)  -\frac{\ga}{N}\sum_{i=1}^N |v_i^t-u^t(x_i^t)|^2\\
+\frac{1}{2\vep^2}\int_{(\R^\d)^2\setminus\triangle} \pa*{u^t(x)-u^t(y)}\cdot\nabla\g(x-y)d\pa*{\mu_N^t-\mu_V-\vep^2\Uu^t}^{\otimes 2}(x,y) \\
+\frac{1}{\vep^2N}\sum_{i=1}^N u^t(x_i^t)\cdot \nabla\zeta(x_i^t)-\int_{\R^\d}\pa*{\div h^{u^t\Uu^t}+h^{\p_t\Uu^t}}d\pa*{\mu_N^t-\mu_V-\vep^2\Uu^t},
\end{multline}
where $h^f\coloneqq \g\ast f$.
\end{lemma}

\begin{proof}[Proof of \cref{lem:MEsmfID}]
The proof is similar to that of \cite[Equation (2.7)]{HkI2021}, but more delicate given that $\mu_V$ is not constant and that $u$ is not a solution to \eqref{eq:Lake}, but rather an extension of the solution to $\R^\d$, which means that $u$ only solves \eqref{eq:Lake} in $\overset\circ\Sigma$. As commented in the introduction, a key difference compared to \cite{HkI2021} is the third term in \eqref{eq:MEcon}, which provides an essential cancellation, as the reader shall see. Finally, we remark that all the computations below are justified by our regularity assumption for $u$. To simplify the notation, we omit the time superscript in the calculations below.

By the chain rule,
\begin{multline}\label{eq:dtHNstart}
\frac{d}{dt}\Hr_{N}(\uz_N,u) = \frac{1}{2N}\sum_{i=1}^N (v_i-u(x_i)) \cdot (\dot{v}_i - \p_t u(x_i) - \nabla u(x_i)\cdot v_i) \\
+ \frac{1}{\vep^2}\frac{d}{dt}\Fr_N(\ux_N,\mu_V+\vep^2\Uu) + \frac{1}{\vep^2N}\sum_{i=1}^N \nabla\zeta(x_i)\cdot \dot{x}_i.
\end{multline}
Let us consider each of the three right-hand side terms separately.

{\bf The third term:} Using the first equation of \eqref{eq:NewODE}  and the definition \eqref{eq:zetadef} of $\zeta$, we find
\begin{equation}\label{eq:dtVcon}
\frac{1}{\vep^2 N}\sum_{i=1}^N \nabla\zeta(x_i)\cdot \dot{x}_i  = \frac{1}{\vep^2 N}\sum_{i=1}^N\nabla \zeta(x_i)\cdot v_i = \frac{1}{\vep^2 N}\sum_{i=1}^N\pa*{\nabla h^{\mu_V}(x_i) + \nabla V(x_i)}\cdot v_i.
\end{equation}


\medskip
{\bf The first term:} Using equation \eqref{eq:NewODE} for $\dot{x}_i, \dot{v}_i$, we find
\begin{multline}
(v_i - u(x_i))\cdot (\dot{v}_i - \p_t u(x_i) - (\dot{x}_i\cdot\nabla)u(x_i)) \\
= \pa*{v_i- u(x_i)} \cdot \Bigg(-\frac{1}{\vep^2 N}\sum_{\substack{1\leq j\leq N \\ j\neq i}}\nabla\g(x_i-x_j)
-\ga v_i- \frac{1}{\vep^2}\nabla V(x_i) -\p_t u(x_i) - (v_i\cdot\nabla) u(x_i)\Bigg).
\end{multline}
Observe that 
\begin{equation}
(u(x_i)\cdot\nabla)u(x_i) - (v_i\cdot\nabla) u(x_i)  = \pa*{(u(x_i)-v_i)\cdot\nabla}u(x_i),
\end{equation}
and therefore,
\begin{align}
(v_i-u(x_i)) \cdot \pa*{(u(x_i)\cdot\nabla)u(x_i) - (v_i\cdot\nabla) u(x_i)} &= (v_i-u(x_i)) \cdot\pa*{(u(x_i)-v_i)\cdot\nabla}u(x_i) \nn\\
&= - \nab u(x_i) : (v_i-u(x_i))^{\otimes 2}.
\end{align}
Similarly,
\begin{equation}
(v_i-u(x_i))\cdot (-\ga v_i) = -\ga |v_i-u(x_i)|^2 - \ga (v_i-u(x_i))\cdot u(x_i).
\end{equation}
Thus,
\begin{multline}\label{eq:dtkecon}
\frac{1}{N}\sum_{i=1}^N \pa*{v_i - u(x_i)}\cdot \pa*{\dot{v}_i - \p_t u(x_i) - \nabla u(x_i)\cdot \dot{x}_i} \\
= \underbrace{-\frac{1}{\vep^2 N^2}\sum_{1\leq j\neq i\leq N}\pa*{v_i - u(x_i)}\cdot \nabla\g(x_i-x_j)}_{T_1} {-\frac{1}{\vep^2 N}\sum_{i=1}^N \pa*{v_i - u(x_i)}\cdot\nabla V(x_i)}\\
- \frac{1}{N}\sum_{i=1}^N \nab u(x_i): (v_i-u(x_i))^{\otimes2} -\frac{\ga}{N}\sum_{i=1}^N |v_i-u(x_i)|^2 \\
{- \frac{1}{N}\sum_{i=1}^N \pa*{v_i - u(x_i)}\cdot\pa*{\p_t u(x_i) + \ga u(x_i) + u(x_i)\cdot\nabla u(x_i)}}.
\end{multline}
As we shall see, the last term will be canceled by the corrector $\Uu$. It is important for the reader to note that we cannot simply write
\begin{equation}
\p_t u +\ga u + u\cdot\nabla u = -\nabla p,
\end{equation}
since we recall that $u$ is actually an extension of the solution to \eqref{eq:Lake}, and therefore the equation only holds in the open set $\overset\circ\Sigma$,  while the particles are not necessarily confined to $\overset\circ\Sigma$.

\medskip
{\bf The second term:} Unpacking the definition \eqref{eq:MEdef} of $\Fr_N$, 
\begin{multline}
\frac{d}{dt} \Fr_N(\ux_N, \mu_V + \vep^2\Uu) = \frac{d}{dt}\frac{1}{2N^2}\sum_{1\leq i\neq j\leq N}\g(x_i-x_j) - \frac{d}{dt}\frac{1}{N}\sum_{i=1}^N h^{\mu_V + \vep^2\Uu}(x_i) \\
 +  \frac{d}{dt} \frac{1}{2}\int_{(\R^\d)^2} \g(x-y) d\pa*{\mu_V+\vep^2\Uu}^{\otimes 2}(x,y).
\end{multline}
By the chain rule and using the first equation of \eqref{eq:NewODE},
\begin{align}
\frac{d}{dt}\frac{1}{2N^2}\sum_{1\leq i\neq j\leq N}\g(x_i-x_j) &= \frac{1}{2N^2}\sum_{1\leq i\neq j\leq N} \nabla\g(x_i-x_j)\cdot\pa*{v_i-v_j} \nn\\
&= \frac{1}{N^2}\sum_{1\leq i\neq j\leq N} \nabla\g(x_i-x_j)\cdot \pa*{v_i - u(x_i)} \nn\\
&\ph+ \frac{1}{N^2}\sum_{1\leq i\neq j\leq N}\nabla\g(x_i-x_j)\cdot u(x_i). \label{eq:dtFcon1}
\end{align}
Since $\mu_V$ is independent of time, it follows from another application of the chain rule that
\begin{equation}
\frac{d}{dt}h^{\mu_V + \vep^2\Uu}(x_i) = \nabla h^{\mu_V+\vep^2\Uu}(x_i)\cdot v_i + \vep^2  h^{\p_t\Uu}(x_i).
\end{equation}
Thus,
\begin{equation}
\begin{split}
- \frac{d}{dt}\frac{1}{N}\sum_{i=1}^N h^{\mu_V + \vep^2\Uu}(x_i) &= -\frac{1}{N}\sum_{i=1}^N v_i\cdot\nabla h^{\mu_V}(x_i)  - \frac{\vep^2}{N}\sum_{i=1}^N v_i \cdot \nabla h^{\Uu}(x_i) - \frac{\vep^2}{N}\sum_{i=1}^N  h^{\p_t\Uu}(x_i). \label{eq:dtFcon2}
\end{split}
\end{equation}
Lastly, by the chain rule and since $\mu_V$ is independent of time,
\begin{equation}\label{eq:dtFcon3}
\frac{d}{dt} \frac{1}{2}\int_{(\R^\d)^2} \g(x-y) d\pa*{\mu_V+\vep^2\Uu}^{\otimes 2}(x,y) = \vep^2\int_{\R^\d} h^{\p_t\Uu}(y)d\pa*{\mu_V+\vep^2\Uu}(y).
\end{equation}
Putting together identities \eqref{eq:dtFcon1}, \eqref{eq:dtFcon2}, \eqref{eq:dtFcon3}, we see that
\begin{multline}
\frac{1}{\vep^2}\frac{d}{dt} \Fr_N(\ux_N, \mu_V + \vep^2\Uu) = \frac{1}{\vep^2 N^2}\sum_{1\leq i\neq j\leq N} \nabla\g(x_i-x_j)\cdot \pa*{v_i - u(x_i)} \\
+ \frac{1}{\vep^2N^2}\sum_{1\leq i\neq j\leq N}\nabla\g(x_i-x_j)\cdot u(x_i) -\frac{1}{\vep^2 N}\sum_{i=1}^N \pa*{v_i-u(x_i)}\cdot\nabla h^{\mu_V}(x_i)  \\
- \frac{1}{N}\sum_{i=1}^N \pa*{v_i-u(x_i)} \cdot \nabla h^{\Uu}(x_i) -\frac{1}{\vep^2 N}\sum_{i=1}^N u(x_i)\cdot\nabla h^{\mu_V}(x_i)- \frac{1}{N}\sum_{i=1}^N u(x_i) \cdot \nabla h^{\Uu}(x_i)\\
 - \frac{1}{N}\sum_{i=1}^N  h^{\p_t\Uu}(x_i) + \int_{\R^\d} h^{\p_t\Uu}d\pa*{\mu_V+\vep^2\Uu}.
\end{multline}
Note that
\begin{equation}
 - \frac{1}{N}\sum_{i=1}^N  h^{\p_t\Uu}(x_i) + \int_{\R^\d}h^{\p_t\Uu}d\pa*{\mu_V+\vep^2\Uu}  =-\int_{\R^\d}h^{\p_t\Uu}d\pa*{\mu_N-\mu_V-\vep^2\Uu}.
\end{equation}
Hence,
\begin{multline}\label{eq:dtFcon'}
\frac{1}{\vep^2}\frac{d}{dt} \Fr_N(\ux_N, \mu_V + \vep^2\Uu) = {\frac{1}{\vep^2 N^2}\sum_{1\leq i\neq j\leq N} \nabla\g(x_i-x_j)\cdot \pa*{v_i - u(x_i)}} \\
+ \frac{1}{\vep^2N^2}\sum_{1\leq i\neq j\leq N}\nabla\g(x_i-x_j)\cdot u(x_i) -\frac{1}{\vep^2 N}\sum_{i=1}^N (v_i-u(x_i))\cdot\nabla h^{\mu_V}(x_i)  \\
- \frac{1}{N}\sum_{i=1}^N (v_i-u(x_i)) \cdot \nabla h^{\Uu}(x_i)-\frac{1}{\vep^2N}\sum_{i=1}^N u(x_i)\cdot\nabla h^{\mu_V}(x_i) -\frac{1}{N}\sum_{i=1}^N u(x_i)\cdot\nabla h^{\Uu}(x_i)  \\
-\int_{\R^\d} h^{\p_t\Uu}d\pa*{\mu_N-\mu_V-\vep^2\Uu}.
\end{multline}
We next reorganize the right-hand side into  into a commutator of the form of the left-hand side of \eqref{main1} plus residual terms that either vanish as $\vep+N^{-1}\rightarrow 0$ or cancel with another term in the computation.

To this end, observe that 
\begin{multline}\label{eq:symprecon}
\frac{1}{2\vep^2}\int_{(\R^\d)^2\setminus\triangle} \pa*{u(x)-u(y)}\cdot\nabla\g(x-y)d\pa*{\mu_N-\mu_V-\vep^2\Uu}^{\otimes 2}(x,y)\\
=\frac{1}{\vep^2 N^2}\sum_{1\leq i\neq j\leq N} u(x_i)\cdot\nabla\g(x_i-x_j)  - \frac{1}{\vep^2 N}\sum_{i=1}^N u(x_i)\cdot\nabla h^{\mu_V}(x_i) - \frac{1}{N}\sum_{i=1}^N u(x_i)\cdot\nabla h^{\Uu}(x_i) \\
-\frac{1}{\vep^2 N}\sum_{i=1}^N\int_{\R^\d} u(x)\cdot\nabla\g(x-x_i)d\pa*{\mu_V+\ep^2\Uu}(x) \\
+ \frac{1}{\vep^2}\int_{\R^\d} \mu_V u\cdot\nabla h^{\mu_V+\vep^2\Uu} +  \int_{\R^\d}u\Uu\cdot \nabla h^{\mu_V+\vep^2\Uu}.
\end{multline}
Since $\mu_V u$ is divergence-free, we see from integrating by parts that
\begin{align}
\frac{1}{\vep^2}\int_{\R^\d} \mu_V u\cdot\nabla h^{\mu_V+\vep^2\Uu} =-\frac{1}{\vep^2}\int_{\R^\d}\div(\mu_V u) h^{\mu_V+\vep^2\Uu} =0.
\end{align}
Similarly, observe that
\begin{align}
&-\frac{1}{\vep^2N}\sum_{i=1}^N\int_{\R^\d}u(x)\cdot\nabla\g(x-x_i)d(\mu_V+\vep^2\Uu)(x) \nn\\
&= \frac{1}{\vep^2 N}\sum_{i=1}^N \int_{\R^\d}\div(\mu_V u)(x)\g(x-x_i)  + \frac{1}{N}  \int_{\R^\d}\div(u\Uu)(x)  \g(  x-x_i)  =\int_{\R^\d}\div h^{u\Uu}d\mu_N \label{eq:huU1}
\end{align}
 and
\begin{align}\label{eq:huU2}
 \int_{\R^\d}u\Uu\cdot \nabla h^{\mu_V+\vep^2\Uu} = -\int_{\R^\d}\div h^{u\Uu} d(\mu_V+\vep^2\Uu).
\end{align}
Thus, we see that the right-hand side of \eqref{eq:symprecon} simplifies,  after combining \eqref{eq:huU1} and \eqref{eq:huU2}, to
\begin{multline}
\frac{1}{\vep^2 N^2}\sum_{1\leq i\neq j\leq N} u(x_i)\cdot\nabla\g(x_i-x_j) - \frac{1}{\vep^2 N}\sum_{i=1}^N u(x_i)\cdot\nabla h^{\mu_V}(x_i)- \frac{1}{N}\sum_{i=1}^N u(x_i)\cdot\nabla h^{\Uu}(x_i) \\
+ \int_{\R^\d} \div h^{u\Uu}d\pa*{\mu_N-\mu_V-\vep^2\Uu},
\end{multline} 
from which it follows that
\begin{multline}\label{eq:dtFcon}
\frac{1}{\vep^2}\frac{d}{dt} \Fr_N(\ux_N, \mu_V + \vep^2\Uu) = \frac{1}{2\vep^2}\int_{(\R^\d)^2\setminus\triangle} \pa*{u(x)-u(y)}\cdot\nabla\g(x-y)d\pa*{\mu_N-\mu_V-\vep^2\Uu}^{\otimes 2}(x,y) \\
\underbrace{+\frac{1}{\vep^2 N^2}\sum_{1\leq i\neq j\leq N} \nabla\g(x_i-x_j)\cdot \pa*{v_i - u(x_i)}}_{-T_1}  -\frac{1}{\vep^2 N}\sum_{i=1}^N (v_i-u(x_i))\cdot\nabla h^{\mu_V}(x_i) - \frac{1}{N}\sum_{i=1}^N (v_i-u(x_i)) \cdot \nabla h^{\Uu}(x_i)  \\
-\int_{\R^\d}\pa*{\div h^{u\Uu}+h^{\p_t\Uu}}d\pa*{\mu_N-\mu_V-\vep^2\Uu}.
\end{multline}

Putting together the identities \eqref{eq:dtVcon}, \eqref{eq:dtkecon}, \eqref{eq:dtFcon} and canceling the $+T_1$ and $-T_{1}$ terms, we arrive at 
\begin{multline}\label{eq:dtHcon}
\frac{d}{dt}\Hr_{N}(\uz_N,u) =  -\frac{1}{N}\sum_{i=1}^N \nab u(x_i) : \pa*{v_i - u(x_i)}^{\otimes2}-\frac{\ga}{N}\sum_{i=1}^N |v_i-u(x_i)|^2\\ 
{+\frac{1}{\vep^2 N}\sum_{i=1}^N\pa*{\nabla h^{\mu_V}(x_i) + \nabla V(x_i)}\cdot v_i {-\frac{1}{\vep^2 N}\sum_{i=1}^N \pa*{v_i - u(x_i)}\cdot\nabla V(x_i)}} {-\frac{1}{\vep^2 N}\sum_{i=1}^N (v_i-u(x_i))\cdot\nabla h^{\mu_V}(x_i)}\\
{- \frac{1}{N}\sum_{i=1}^N \pa*{v_i - u(x_i)}\cdot\pa*{\p_t u(x_i) + \ga u(x_i) + (u(x_i)\cdot\nabla) u(x_i)}}  {- \frac{1}{N}\sum_{i=1}^N (v_i-u(x_i)) \cdot \nabla h^{\Uu}(x_i)} \\
+\frac{1}{2\vep^2}\int_{(\R^\d)^2\setminus\triangle} \pa*{u(x)-u(y)}\cdot\nabla\g(x-y)d\pa*{\mu_N-\mu_V-\vep^2\Uu}^{\otimes 2}(x,y) \\
 -\int_{\R^\d}\pa*{\div h^{u\Uu}+h^{\p_t\Uu}}d\pa*{\mu_N-\mu_V-\vep^2\Uu}.
\end{multline}

Let us simplify the right-hand side. Recalling the definition \eqref{eq:zetadef} of $\zeta$, we see that the second line reduces to
\begin{equation}
\frac{1}{\vep^2N}\sum_{i=1}^N u(x_i)\cdot \pa*{\nabla V(x_i) + \nabla h^{\mu_V}(x_i)} = \frac{1}{\vep^2N}\sum_{i=1}^N u(x_i)\cdot \nabla\zeta(x_i).
\end{equation}
We also want the third line to cancel. Accordingly, we choose $\Uu$ so that
\begin{equation}\label{eq:nabhUu}
-\nabla h^{\Uu} = \p_t u+\ga u+u\cdot\nabla u.
\end{equation}
Applying $\div$ to both sides and using  that $h^{f} = \cd (-\Delta)^{\frac{\s-\d}{2}}f$ and the Fourier multiplier identity $-\div\nabla(-\D)^{\frac{\s-\d}{2}} =  (-\D)^{\frac{2+\s-\d}{2}}$, we see that we should choose $\Uu = (-\D)^{\frac{\d-\s-2}{2}}\div(\p_t u+\ga u+u\cdot\nabla u)$, which evidently has zero mean. With this final cancellation, the proof of the lemma is now complete.
\end{proof}

Let us close this subsection with some comments on how to modify the preceding calculations when $\uz_N^t$ is now a solution of the system of SDEs \eqref{eq:NewSDE}. The argument is similar to that in the proofs of \cite[Lemma 6.1, Proposition 6.3]{RS2021}. One should start with a solution $\uz_{N,\ep}$ of a truncated SDE system, where the interaction potential $\g_{(\ep)}$ is smooth and agrees with $\g$ at distance $\ep$ from the origin.  Consider a modified version $\mathsf{H}_{N,\ep}$ of the total modulated energy $\mathsf{H}_N$, where $\g$ in $\Fr_N(X_{N,\ep},\mu_V+\vep^2\Uu)$ has been replaced by $\g_{(\ep)}$. One then computes $\E[\Hr_{N,\ep}(\uz_{N,\ep}^{\tau_\ep}, u^{\tau_{\ep}})]$, where $\tau_\ep$ is the stopping time when the minimal distance between particles is $\le \vep$, using Doob's optional stopping theorem and then shows that the truncated potential $\g_{(\ep)}$ can be replaced by $\g$ as $\ep\rightarrow 0$, up to vanishing error terms. The main difference between the setting of \eqref{eq:NewSDE} and that of \eqref{eq:NewODE} is that, by It\^o's formula, the nontrivial quadratic variation of the Wiener processes contributes a new term $+\d/\be$, coming from the modulated kinetic energy, to the right-hand side of \eqref{eq:dtHNstart}. Evidently, this new term vanishes if $\be\rightarrow\infty$ as $N\rightarrow\infty$.

\subsection{Gr\"{o}nwall argument}\label{ssec:MPgron}
We are in a position to complete the proof of \cref{thm:mainSMFrig}. With the formula \eqref{eq:dtHNid} given by \cref{lem:MEsmfID}, this part is just an application of estimates already in our possession.

Trivially,
\begin{multline}\label{eq:LGronkin}
-\frac{1}{N}\sum_{i=1}^N \nab u(x_i) : \pa*{v_i - u(x_i)}^{\otimes 2}  -\frac{\ga}{N}\sum_{i=1}^N |v_i-u(x_i)|^2 \\
\leq \frac{(\|(\nabla u)_{sym,-}\|_{L^\infty}-\ga)}{N}\sum_{i=1}^N |u(x_i)-v_i|^2
\end{multline}
where $(A)_{sym,-}$ denotes the negative, symmetric part of a square matrix $A$ (note that the antisymmetric part of $\nab u(x_i)$ has zero contribution above). This takes care of the first line on the right-hand side of \eqref{eq:dtHNid}. 

The second line in \eqref{eq:dtHNid} is handled via the sharp commutator estimate of \cref{prop:comm} with $\mu$ replaced by $\mu_V+\vep^2\Uu$ (note that $\int_{\R^\d}(\mu_V+\vep^2\Uu) = 1$ since $\mu_V$ is a probability measure and $\Uu$ has zero mean):
\begin{multline}\label{eq:LGronpot}
\frac{1}{2\vep^2}\Big|\int_{(\R^\d)^2\setminus\triangle} \pa*{u(x)-u(y)}\cdot\nabla\g(x-y)d\pa*{\mu_N-\mu_V-\vep^2\Uu}^{\otimes 2}(x,y) \Big| \\
 \leq \frac{C\|\nabla u\|_{L^\infty}}{\vep^2}\pa*{\Fr_N (\ux_N,\mu) +  \frac{\log(N\|\mu\|_{L^\infty})}{2\d N}\indic_{\s=0} +  C \|\mu\|_{L^\infty}^{\frac{\s}{\d}} N^{-1+\frac{\s}{\d}}}.
\end{multline}
By \cref{lem:ugradVext},
\begin{equation}\label{eq:LGronzeta}
\frac{1}{\vep^2N}\sum_{i=1}^N|u(x_i)\cdot \nabla\zeta(x_i)| \leq \frac{C\|u\|_{W^{1,\infty}}}{\vep^2N}\sum_{i=1}^N \zeta(x_i).
\end{equation}
Finally,  for the third line in \eqref{eq:dtHNid},  by \cref{lem:MPEcoer} with $\phi = \div h^{u\Uu}+h^{\p_t\Uu}$, we have that
\begin{multline}\label{eq:LGronRem}
\Big|\int_{\R^\d}\pa*{\div h^{u\Uu}+h^{\p_t\Uu}}d\pa*{\mu_N-\mu_V-\vep^2\Uu}\Big| \\
 \leq C\pa*{\|(-\Delta)^{\frac{\d-\s}{2}}\div h^{u\Uu}\|_{L^\infty}+\|(-\Delta)^{\frac{\d-\s}{2}}h^{\p_t\Uu}\|_{L^\infty}}(N\|{\mu_V+\vep^2\Uu}\|_{L^\infty})^{\frac{\s}{\d}-1} \\
+ C\pa*{\|\nabla\div h^{u\Uu}\|_{\dot{H}^{\frac{\d-\s}{2}}} + \|\nabla h^{\p_t\Uu}\|_{\dot{H}^{\frac{\d-\s}{2}}}}\Big(\Fr_N(\ux_N,\mu_V+\vep^2\Uu) + \frac{\log(\|\mu_V+\vep^2\Uu\|_{L^\infty}N)}{2\d N}\indic_{\s=0} \\
+ C\|\mu_V+\vep^2\Uu\|_{L^\infty}^{\frac{\s}{\d}}N^{\frac{\s}{\d}-1}\Big)^{1/2},
\end{multline}
where we have also implicitly used the triangle inequality.
{Applying the elementary inequality $ ab \le \frac12 (\vep^2 a^2+\frac{b^2}{ \vep^2})$, we may also bound the second term on the right-hand side by 
\begin{multline}\label{eq:prelog}
C\vep^2\pa*{\|\nabla\div h^{u\Uu}\|_{\dot{H}^{\frac{\d-\s}{2}}} + \|\nabla h^{\p_t\Uu}\|_{\dot{H}^{\frac{\d-\s}{2}}}}^2 \\
+ \frac{C}{\vep^2}\Big(\Fr_N(\ux_N,\mu_V+\vep^2\Uu) + \frac{\log(\|\mu_V+\vep^2\Uu\|_{L^\infty}N)}{2\d N}\indic_{\s=0}+ C\|\mu_V+\vep^2\Uu\|_{L^\infty}^{\frac{\s}{\d}}N^{\frac{\s}{\d}-1}\Big).
\end{multline}
Putting together the estimates \eqref{eq:LGronkin}, \eqref{eq:LGronpot}, \eqref{eq:LGronzeta}, \eqref{eq:LGronRem}, and \eqref{eq:prelog}, we have shown that
\begin{multline}
\frac{d}{dt}\Hr_{N}(\uz_N,u)  \leq \frac{(\|(\nabla u)_{sym,-}\|_{L^\infty}-\ga)}{N}\sum_{i=1}^N |u(x_i)-v_i|^2 \\
+\frac{C\|\nabla u\|_{L^\infty}}{\vep^2}\pa*{\Fr_N (\ux_N,\mu_V+\vep^2\Uu) +  \frac{\log(N\|\mu_V+\vep^2\Uu\|_{L^\infty})}{2\d N}\indic_{\s=0} +  C \|\mu_V+\vep^2\Uu\|_{L^\infty}^{\frac{\s}{\d}} N^{-1+\frac{\s}{\d}}}\\
+ \frac{C\|u\|_{W^{1,\infty}}}{\vep^2N}\sum_{i=1}^N \zeta(x_i)+ C\pa*{\|(-\Delta)^{\frac{\d-\s}{2}}\div h^{u\Uu}\|_{L^\infty}+\|(-\Delta)^{\frac{\d-\s}{2}}h^{\p_t\Uu}\|_{L^\infty}}(N\|{\mu_V+\vep^2\Uu}\|_{L^\infty})^{\frac{\s}{\d}-1}\\
+
C\vep^2\pa*{\|\nabla\div h^{u\Uu}\|_{\dot{H}^{\frac{\d-\s}{2}}}+ \|\nabla h^{\p_t\Uu}\|_{\dot{H}^{\frac{\d-\s}{2}}}}^2 + \frac{C}{\vep^2}\Big(\Fr_N(\ux_N,\mu_V+\vep^2\Uu) + \frac{\log(\|\mu_V+\vep^2\Uu\|_{L^\infty}N)}{2\d N}\indic_{\s=0}\\
+ C\|\mu_V+\vep^2\Uu\|_{L^\infty}^{\frac{\s}{\d}}N^{\frac{\s}{\d}-1}\Big).
\end{multline}

After a little bookkeeping, rewriting our differential inequality in integral form using the fundamental theorem of calculus, we have shown that
\begin{multline}\label{eq:Hpreadd}
\Hr_{N}(\uz_N^t,u^t) \leq \Hr_{N}(\uz_N^0,u^0) + \int_0^t\frac{(\|(\nabla u^\tau)_{sym,-}\|_{L^\infty}-\ga)}{N}\sum_{i=1}^N |u^\tau(x_i^\tau)-v_i^\tau|^2 d\tau \\
+ \int_0^t\Bigg(\frac{C\|u^\tau\|_{W^{1,\infty}}}{\vep^2N}\sum_{i=1}^N \zeta(x_i^\tau)+ C\pa*{\|(-\Delta)^{\frac{\d-\s}{2}}\div h^{u^\tau\Uu^\tau}\|_{L^\infty}+\|(-\Delta)^{\frac{\d-\s}{2}}h^{\p_\tau\Uu^\tau}\|_{L^\infty}} \\
\times(N\|{\mu_V+\vep^2\Uu^\tau}\|_{L^\infty})^{\frac{\s}{\d}-1}\Bigg) d\tau+\int_0^t\frac{C(1+\|\nabla u^\tau\|_{L^\infty})}{\vep^2}\Bigg(\Fr_N (\ux_N^\tau,\mu_V+\vep^2\Uu^\tau) +  \frac{\log(N\|\mu_V+\vep^2\Uu^\tau\|_{L^\infty})}{2\d N}\indic_{\s=0} \\
+  C \|\mu_V+\vep^2\Uu^\tau\|_{L^\infty}^{\frac{\s}{\d}} N^{-1+\frac{\s}{\d}}\Bigg)d\tau +\int_0^t C\vep^2\pa*{\|\nabla\div h^{u^\tau\Uu^\tau}\|_{\dot{H}^{\frac{\d-\s}{2}}} + \|\nabla h^{\p_\tau\Uu^\tau}\|_{\dot{H}^{\frac{\d-\s}{2}}}}^2d\tau.
\end{multline}
Now adding 
\begin{equation}
\frac{\log(N\|\mu_V+\vep^2\Uu^t\|_{L^\infty})}{2\d N\vep^2}\indic_{\s=0} + C\frac{\|\mu_V+\vep^2\Uu^t\|_{L^\infty}^{\frac{\s}{\d}}N^{\frac{\s}{\d}-1}}{\vep^2}
\end{equation}
to both sides of \eqref{eq:Hpreadd} so that the left-hand side is nonnegative provided $C$ is sufficiently large depending on $\d,\s$ (recall \cref{lem:MPElb}), it follows now from an application of the Gr\"{o}nwall-Bellman lemma that
\begin{equation}
\mathscr{H}_{N}^t \leq B^t e^{A^t},
\end{equation}
where
\begin{equation}
A^t\coloneqq C\int_0^t \pa*{1+\|\nabla u^\tau\|_{L^\infty} + \|u^\tau\|_{W^{1,\infty}} + (\|(\nabla u^\tau)_{sym,-}\|_{L^\infty}-\ga)_+}d\tau
\end{equation}
and 
\begin{multline}
B^t \coloneqq {\mathscr{H}_{N}(\uz_N^0,u^0)}\\
 + \Bigg(\sup_{0\le\tau\le t}\Big(\frac{\log(\frac{\|\mu_V+\vep^2\Uu^\tau\|_{L^\infty}}{\|\mu_V+\vep^2\Uu^0\|_{L^\infty}})}{2\d N\vep^2}\indic_{\s=0} + \frac{C(\|\mu_V+\vep^2\Uu^\tau\|_{L^\infty}^{\frac{\s}{\d}} -\|\mu_V+\vep^2\Uu^0\|_{L^\infty}^{\frac{\s}{\d}}) N^{\frac{\s}{\d}-1}}{\vep^2}\Big)\Bigg)_+ \\
+ \int_0^t  C\pa*{\|(-\Delta)^{\frac{\d-\s}{2}}\div h^{u^\tau\Uu^\tau}\|_{L^\infty}+\|(-\Delta)^{\frac{\d-\s}{2}}h^{\p_\tau\Uu^\tau}\|_{L^\infty}}(N\|{\mu_V+\vep^2\Uu^\tau}\|_{L^\infty})^{\frac{\s}{\d}-1}d\tau \\
+\int_0^t C\vep^2\pa*{\|\nabla\div h^{u^\tau\Uu^\tau}\|_{\dot{H}^{\frac{\d-\s}{2}}} + \|\nabla h^{\p_\tau\Uu^\tau}\|_{\dot{H}^{\frac{\d-\s}{2}}}}^{ 2}d\tau.
\end{multline}

The various norms involving the $h$'s can be estimated in terms of norms of $u^\tau,\Uu^\tau,\p_\tau\Uu^\tau$ by means of routine potential-theoretic estimates (cf. \cite[Lemmas 4.3-4.4]{Rosenzweig2021ne}). Indeed, by Plancherel's theorem, $\|\nabla h^{\p_\tau\Uu^\tau}\|_{\dot{H}^{\frac{\d-\s}{2}}} \leq \|\p_\tau\Uu^\tau\|_{\dot{H}^{\frac{2+\s-\d}{2}}}$ and
\begin{equation}
\|\nabla\div h^{u^\tau \Uu^\tau}\|_{\dot{H}^{\frac{\d-\s}{2}}} \leq \|(u^\tau\Uu^\tau)\|_{\dot{H}^{\frac{\s+4-\d}{2}}} \leq C\pa*{\|u^\tau\|_{\dot{H}^{\frac{\s+4-\d}{2}}}\|\Uu^\tau\|_{L^\infty} + \|u^\tau\|_{L^\infty}\|\Uu^\tau\|_{\dot{H}^{\frac{\s+4-\d}{2}}}},
\end{equation}
where the ultimate inequality follows from the homogeneous Moser product estimate \cite[Corollary 2.54]{BCD2011} (note $\s+4-\d\ge 2$).  Similarly, using the fact that $(-\Delta)^{\frac{\d-\s}{2}}h^f = \cd f$ and the product rule,
\begin{align}
\|(-\Delta)^{\frac{\d-\s}{2}}\div h^{u^\tau\Uu^\tau}\|_{L^\infty} \le C\|\div(u\Uu)\|_{L^\infty} \le C(\|\div u^\tau\|_{L^\infty}\|\Uu^\tau\|_{L^\infty} + \|u\|_{L^\infty}\|\nab \Uu^\tau\|_{L^\infty})
\end{align}
and
\begin{align}
\|(-\Delta)^{\frac{\d-\s}{2}}h^{\p_\tau\Uu^\tau}\|_{L^\infty} \le C\|\p_\tau\Uu^\tau\|_{L^\infty}.
\end{align}

The proof of \cref{thm:mainSMFrig} is now complete. One can go further and obtain a self-contained bound in terms of Sobolev norms of $u^t$, which in turn can be controlled by Sobolev norms of the original solution to equation \eqref{eq:Lake} (recall that $u$ is really the extension to all of $\R^\d$). We leave these details to the interested reader.

\section{The regime $\vep^2 \lesssim N^{\frac{\s}{\d}-1}$ for the 1D Coulomb gas}\label{sec:1DCou}

In this section, we consider the one-dimensional Coulomb case with quadratic confinement (i.e. $\s=-1$, $\g(x)= - |x|$, $V(x)=|x|^2$). Under the assumption that the initial empirical spatial density converges to the equilibrium measure, we show there is in general no weak limit for the empirical current as $\vep + N^{-1}\rightarrow 0$ if $\vep^{-2} N^{\frac{\s}{\d}-1} = \vep^{-2} N^{-2}$ does not vanish. Thus, for the supercritical mean-field limit in the (super-)Coulomb Riesz case, one cannot expect convergence to the Lake equation outside of the scaling assumption \eqref{eq:SMFscl} of our \cref{thm:mainSMF}. 

\begin{prop}\label{prop:1DCou}
There exists a sequence of initial positions $\XN^\circ$ such that $\Fr_N(\XN^\circ,\mu_V)\rightarrow 0$  and $\mu_N^\circ \rightharpoonup \mu_V$ as $N\rightarrow\infty$, but that if {$v_i^\circ = 0$ for each $1\le i\le N$},  the associated empirical current $J_N^t \coloneqq \frac1N\sum_{i=1}^N v_i^t\delta_{x_i^t}$ 
\begin{itemize}
\item converges to zero (which is the unique solution of the Lake equation \eqref{eq:LakeJ} {starting from zero initial datum}) uniformly on $[0,\infty)$ as $\vep+N^{-1}\rightarrow 0$  if $\vep N\rightarrow \infty$,
\item has no weak limit for any $t\in (0,\infty)$ as  $\vep+N^{-1}\rightarrow 0$ if $\vep N \not\rightarrow \infty$.
\end{itemize}
\end{prop}
\begin{proof}
As is folklore, the dynamics of the one-dimensional Coulomb gas with quadratic confinement are exactly solvable. Without loss of generality, we assume that the initial positions are ordered $x_1^\circ<\cdots < x_N^\circ$. Observe that for each $1\le i\le N$,
\begin{align}
- \frac1N\sum_{j\ne i}\g'(x_i-x_j) = \frac2N\sum_{j\neq i}\sgn(x_i-x_j) = \frac2N\sum_{j<i} 1 - \frac2N\sum_{j>i}1 &= \frac{2[(i-1) - (N-i)]}{N} \nn\\
& = \frac{2(2i - 1-N)}{N}.
\end{align}
Let $T_*$ be the maximal time such that $x_1^t<\cdots<x_N^t$ for all $t<T_*$. Such a $T_*$ exists by the initial point separation and the continuity of trajectories. Then on $[0,T^*)$, we have that
\begin{align}
\vep^2 \ddot{x}_i^t = - 2x_i^t +  \frac{2(2i - 1-N)}{N},
\end{align}
which may be integrated on $[0,T^*)$ to yield
\begin{align}\label{eq:xi1Dcou}
x_i^t = \Big(x_i^\circ -  \frac{(2i - 1-N)}{N}\Big)\cos(\sqrt{2}t/\vep) + v_i^\circ\sin(\sqrt{2}t/\vep) + \frac{(2i - 1-N)}{N}.
\end{align}
Hence,
\begin{align}
x_{i+1}^t - x_i^t = (v_{i+1}^\circ - v_i^\circ)\sin(\sqrt{2}t/\vep) + (x_{i+1}^\circ-x_i^\circ - \frac{2}{N})\cos(\sqrt{2}t/\vep) + \frac{2}{N}.
\end{align}
If for each $i$, we assume that $|v_{i+1}^\circ-v_i^\circ| + |x_{i+1}^\circ-x_i^\circ - \frac2N|< \frac2N$, then it follows from the triangle inequality that $x_{i+1}^t - x_i^t>0$. Thus, under this assumption,  we may conclude from a continuity argument that the particles preserve their initial order for all times and the equation \eqref{eq:xi1Dcou} holds on $[0,\infty)$.

The oscillations on the fast time scale $\vep$, which do not vanish as $\vep\rightarrow 0$, rule out any convergent behavior at the microscopic level. Moreover, we cannot even guarantee the convergence of the empirical current $J_N^t$ to a solution of the Lake equation even when the empirical spatial density $\mu_N^t$ converges to the equilibrium measure $\mu_V = \frac12\indic_{[-1,1]}$.

To see this last assertion, simply take $v_i^\circ = 0$ and $x_i^\circ = \frac{2i-1-N}{N} + \frac1N$ for each $i$. As $(\frac{2i-1-N}{N})_{i=1}^N$ is the unique {(up to permutation)} critical point configuration of the microscopic energy \eqref{eq:ENdef}, $\XN^\circ$ is not a critical point. However, it is straightforward to check that $\frac1N\sum_{i=1}^N\delta_{x_i^\circ} \rightharpoonup \mu_V$ as $N\rightarrow\infty$, and evidently,
\begin{align}
|v_{i+1}^\circ-v_i^\circ|  + |x_{i+1}^\circ-x_i^\circ-\frac2N| = 0 < \frac2N,
\end{align}
so that equation \eqref{eq:xi1Dcou} holds on $[0,\infty)$. Hence,
\begin{align}
x_i^t &= \frac1N\cos(\sqrt{2}t/\vep) + \frac{(2i-1-N)}{N}, \\
v_i^t &= -\frac{\sqrt{2}}{N\vep}\sin(\sqrt{2}t/\vep),
\end{align}
from which it follows that
\begin{align}
J_N^t = \frac1N\sum_{i=1}^N \Big(-\frac{\sqrt{2}\sin(\sqrt{2}t/\vep)}{N\vep}\Big)\delta_{x_i^t} = -\frac{\sqrt{2}\sin(\sqrt{2}t/\vep)}{N\vep}\mu_N^t,
\end{align}
which evidently has no weak limit {for $t>0$} if $N\vep \not\rightarrow \infty$.
\end{proof}


\appendix
\section{Regular interactions}\label{sec:appreg}
In this appendix, we show how to extend our method to treat the case of sufficiently regular interactions on the flat torus $\T^\d$---for which no commutator estimate of the form \cref{prop:comm} can hold in general\footnote{It is not difficult to construct an explicit counterexample when $\g$ is Gaussian.}---under the assumption that the equilibrium measure has full support in $\T^\d$.  The reason for this latter assumption is that there is no concern about boundaries of support and need to connect with the obstacle problem. We recall from \cref{sec:EMOP} that one can tailor the confinement so that the equilibrium measure coincides with a desired probability measure. Hence, the full support assumption is not vacuous. The results of this appendix provide a mathematical basis (in the monokinetic regime) for the universality of the Lake equation as a supercritical mean-field limit for smooth interactions. The motivation for considering $\T^\d$, as opposed to $\R^\d$, is explained in \cref{ssec:appregLake} below. {As elaborated on in \cref{ssec:appregMF} below, this extension of the modulated-energy method works just as well---and is even simpler---for the usual mean-field limit for the second-order monokinetic or first-order dynamics.}

We assume that the interaction $\g$ has zero average and the  confinement $V$ is such that the equilibrium measure $\mu_V$ has full support in $\T^\d$, i.e.~there exists a constant $c\in\R$ such that
\begin{align}\label{eq:equchar'}
h^{\mu_V} + V  = c, \qquad \text{q.e.} \  \in \T^\d,
\end{align}
or equivalently, the function $\zeta$ from \eqref{eq:zetadef} is zero q.e.~in $\T^\d$. We assume that the Fourier transform $\hat{\g}(\xi)>0$ for $\xi\ne 0$ (a type of repulsive assumption), which implies by \eqref{eq:equchar'} that
\begin{align}
\mu_V -1 = {\widehat{(1/\g)}(D)}\Big(c-V\Big), \qquad \text{a.e.} \ \in \T^\d,
\end{align}
where {$\widehat{(1/\g)}(D)$} is the Fourier multiplier with symbol {$1/\hat{\g}(\xi)$}.

Let $\w(x-y)$ be the kernel of the Fourier multiplier $(I-\Delta)^{-\frac{\ka}{2}}$ for some {$\ka > \d$}, {the exact requirement for which will be specified below}. Clearly, $\w$ is positive definite and by Sobolev embedding, $\w\in C^{{\ka-\d}}$ if {$\ka>\d$}. Observe that if $\tl{h}^f \coloneqq \w\ast f$, then we may rewrite
\begin{align}\label{eq:welec}
 \|f\|_{H^{-\frac{\ka}{2}}}^2 = \int_{(\T^\d)^2}\w(x-y)df^{\otimes 2}(x,y)  = \int_{\T^\d} |(I-\Delta)^{\frac{\ka}{4}}\tl{h}^{f}|^2.
\end{align}

We {redefine} the total modulated energy as
\begin{multline}\label{eq:TMEdefreg}
\Hr_N(\ZN^t,u^t) \coloneqq \frac{1}{2N}\sum_{i=1}^N |v_i^t-u^t(x_i^t)|^2 + \frac{1}{\vep^2}\Fr_N(\ux_N^t,\mu_V+\vep^2\Uu^t)  \\
+\frac{1}{2\vep^2}\int_{(\T^\d)^2}\w(x-y)d(\mu_N^t-\mu_V-\vep^2\Uu^t)^{\otimes 2}(x,y)
\end{multline}
for a solution $\ZN^t$ of the microscopic system \eqref{eq:NewODE} and a solution $u^t$ of the Lake equation \eqref{eq:Lake}. We no longer need to consider an extension $\tl{u}^t$ of a solution to \eqref{eq:Lake} as in \cref{sec:MP} because $\mu_V$ has full support. Also, note that the $\zeta$ term has disappeared because $\zeta\equiv 0$, as remarked above. Strictly speaking, the correction $\vep^2\Uu^t$ in the last term of \eqref{eq:TMEdefreg} is not necessary; but it leads to cleaner estimates. We also mention that $\Uu^t$ is now explicitly given in terms of the pressure by $h^{\Uu^t} = p^t$, which may be inverted by applying {$\widehat{1/\g}(D)$} to both sides, using our assumption that $\hat{\g}(\xi)>0$ for $\xi\ne 0$. 

We assume further that $\g$ has the following regularity property: {for some $\ka>\d+2$,} there exists a constant $C>0$ such that for any test function $f$,
\begin{align}\label{eq:nabgHka}
\|\nab\g\ast f\|_{H^{\frac{\ka}{2}}} \le C \|f\|_{H^{-\frac{\ka}{2}}}.
\end{align}
{We choose $\w$ based on this value of $\ka$.} As the reader may check from taking $f$ to be a smearing of $\delta_0$, the condition \eqref{eq:nabgHka} implies that $\nab\g$ is {$C^1$}. Consequently, we may re-insert the diagonal in the definition \eqref{eq:MEdef} of the modulated potential energy, yielding a nonnegative quantity {that we continue to denote by $\Fr_N$}. It is evident from Plancherel's theorem that this property is satisfied, for instance, if the Fourier transform of $\g$ is bounded and is $\lesssim |\xi|^{-1-\ka}$ as $|\xi|\rightarrow\infty$. 


The main result of this appendix is the following theorem (cf. \cref{thm:mainSMFrig} above).

\begin{thm}\label{thm:mainSMFrigreg}
Let $\ka,\g,\mu_V$ satisfy the above assumptions. Let $Z_N^t$ be a solution \eqref{eq:NewODE} and $u^t$ be a solution of \eqref{eq:Lake}. Then there {exists a constant $C>0$} depending only on {$\d,\s,\ka$ and the constant in \eqref{eq:nabgHka}}, such that for every $t\in [0,T]$,
\begin{multline}\label{eq:SMFrigreg}
\Hr_{N}(\uz_N^t,u^t) \le \exp\Big(C\int_0^t(\|u^\tau\|_{H^{\frac{\ka+2}{2}}} + (\|(\nab u^\tau)_{sym,-} \|_{L^\infty}+\frac12-\ga)_+ ) d\tau\Big)\\
\times\Bigg(\Hr_{N}(\uz_N^0,u^0) + \vep^2\int_0^t  \Big( \|u^\tau\Uu^\tau\|_{H^{-\frac{\ka}{2}}} + \|(-\Delta)^{-\frac{1}{2}}\p_\tau \Uu^\tau\|_{H^{-\frac{\ka}{2}}}\Big)^2  d\tau\Bigg).
\end{multline}
\end{thm}

\subsection{Regularity for the Lake equation}\label{ssec:appregLake}
{Before turning to the proof of \cref{thm:mainSMFrigreg}, let us comment on the regularity assumptions for the solution $u$ of the Lake equation \eqref{eq:Lake} and why we choose to consider $\T^\d$ instead of $\R^\d$.


By testing the identity \eqref{eq:press} against $p$, integrating by parts, and using Cauchy-Schwarz, we see that
\begin{align}
\int_{\T^\d}|\nab p|^2 d\mu_V \le \int_{\T^\d}\nab p (u\cdot\nab)u d\mu_V \le \|\nab p\|_{L^2(\mu_V)} \|u\|_{L^\infty} \|\nab u\|_{L^2(\mu_V)},
\end{align}
which  implies that
\begin{align}\label{eq:nabp}
\|\nab p\|_{L^2(\mu_V)} \le \|u\|_{L^\infty}\|\nab u\|_{L^2(\mu_V)}.
\end{align}
In particular, if $\inf\mu_V >0$, then we have an $L^2$ bound for $\nab p$ in terms of $\|u\|_{L^\infty},\|\nab u\|_{L^2(\mu_V)}$. In the whole space, it is not possible to have a probability density that is uniformly bounded from below. To avoid this issue in $\R^\d$, we would then have to work in weighted spaces. But this would require modification of our modulated energy scheme in the form of the third term of \eqref{eq:TMEdefreg}. Furthermore, \eqref{eq:nabp} on its own is not useful. It will not allow us to close an energy estimate for $\|u\|_{L^2(\mu_V)}$ due to the loss of derivative and the need for an $L^\infty$ bound on $u$. Instead, we want to show---and indeed can show using the $\div(\mu_Vu)=0$ assumption to overcome the loss of derivative---that for large enough $k$ (namely, $k>\d$), $\|\nab^{\otimes(k+1)}p\|_{L^2(\mu_V)}$ is controlled by $\sum_{j=0}^k\|\nab^{\otimes j} u\|_{L^2(\mu_V)}$ and $\|u\|_{C^{\frac{k}{2}}}$, up to a factor depending on the $L^\infty$ norm of derivatives of $\mu_V$ and $(\mu_V)^{-1}$. Assuming suitable decay assumptions on $\mu_V$ and its derivatives, this latter factor can even be bounded in the whole space. Such a bound for $\|\nab^{\otimes(k+1)}p\|_{L^2(\mu_V)}$ is acceptable for an energy estimate for $\sum_{j=0}^k\|\nab^{\otimes j}u\|_{L^2(\mu_V)}$, provided that one can control $\|u\|_{C^{\frac{k}{2}}}$. Unfortunately, such a control does not follow from $\sum_{j=0}^k\|\nab^{\otimes j}u\|_{L^2(\mu_V)}$ unless $\inf\mu_V>0$. In which case, the assumption on $k$ may weakened to $k>\frac{\d}{2}+1$ through a more efficient use of H\"older and Sobolev inequalities.

For the above described reasons, we work on the torus, where the regularity assumptions on $u$ in \cref{thm:mainSMFrigreg} may be shown to hold if $\mu_V$ is sufficiently smooth and $\inf\mu_V>0$. The argument is similar to the case of the bounded domain, and, in fact, is easier due to the absence of the boundary condition. We emphasize, though, that the proof of \cref{thm:mainSMFrigreg} presented in the next subsection goes through unchanged (removing the zero mean assumption on $\g$) if we instead work in the whole space, and \cref{thm:mainSMFrigreg} is valid in $\R^\d$, conditional on the regularity for $u$. We are not aware of a work establishing solutions with this regularity (existing results in the whole space, such as \cite{Duerinckx2018} for $\d=2$, assume $\log\mu_V \in L^\infty$), and it would be interesting to address this.

\subsection{Main proof}
To prove \cref{thm:mainSMFrigreg}, we first compute the evolution of the new total modulated energy \eqref{eq:TMEdefreg} (cf. \cref{lem:MEsmfID} above).

\begin{lemma}\label{lem:MEsmfIDreg}
It holds that
\begin{multline}\label{eq:MEsmfIDreg}
\frac{d}{dt}\Hr_{N}(\uz_N^t,u^t) = -\frac{1}{N}\sum_{i=1}^N \nab u^t(x_i^t) : \pa*{v_i^t - u^t(x_i^t)}^{\otimes 2} -\frac{\ga}{N}\sum_{i=1}^N |v_i^t-u^t(x_i^t)|^2\\
+\frac{1}{2\vep^2}\int_{(\T^\d)^2\setminus\triangle} \pa*{u^t(x)-u^t(y)}\cdot\nabla\g(x-y)d(f^t)^{\otimes 2}(x,y) -\int_{\T^\d}\pa*{\div h^{u^t\Uu^t}+h^{\p_t\Uu^t}}df^t\\
+\frac{1}{N\vep^2}\sum_{i=1}^N (v_i^t-u^t(x_i^t))\cdot(\nab\tl{h}^{f^t})(x_i^t) + \frac{1}{\vep^2}\int_{\T^\d}u^t\cdot\nab \tl{h}^{f^t} d{f^t} - \int_{\T^\d} \Big(\div \tl{h}^{u^t\Uu^t}df^t +\tl{h}^{\p_t\Uu^t}\Big) df^t,
\end{multline}
\end{lemma}
where  we abbreviate $f^t\coloneqq \mu_N^t - \mu_V-\vep^2\Uu^t$.
\begin{proof}
{We omit the time superscripts in the computations below.}

By the product rule and equation \eqref{eq:NewODE} for $\dot{x}_i$, we find
\begin{align}
&\frac12\frac{d}{dt}\int_{(\T^\d)^2}\w(x-y)df^{\otimes 2}(x,y)\nn\\
&= \frac1N\sum_{i=1}^N\int_{(\T^\d)^2}v_i\cdot \nab\w(x_i-y)df(y) - \int_{(\T^\d)^2} \p_t\Uu(x) \w(x-y)df(y)\nn\\
&=\frac1N\sum_{i=1}^N (v_i-u(x_i))\cdot (\nab\tl{h}^{f})(x_i)  + \int_{\T^\d} u\cdot\nab\tl{h}^f df + \int_{\T^\d}u\cdot\nab\tl{h}^f d\mu_V\nn\\
&\ph   - \int_{\T^\d} \div \tl{h}^{u\Uu}df -\int_{\T^\d}\tl{h}^{\p_t\Uu} df . \label{eq:dtwff}
\end{align}
Since $\div(\mu_V u) = 0$, Fubini-Tonelli and integration by parts in $x$ reveal that the last term on the third line equals zero. 
Combining with \cref{lem:MEsmfID}, we arrive at the desired \eqref{eq:MEsmfIDreg}. 
\end{proof}

We now estimate each term on the right-hand side of \eqref{eq:MEsmfIDreg} and establish a Gr\"onwall relation to prove \cref{thm:mainSMFrigreg}.

\begin{proof}[Proof of \cref{thm:mainSMFrigreg}]

The only terms on the right-hand side of \eqref{eq:MEsmfIDreg} that require modification in estimating, compared to as in \cref{ssec:MPgron}, are the third through sixth.

{\bf The third term:} Instead of using the commutator estimate of \cref{prop:comm}, we argue as follows. Desymmetrizing,
\begin{align}
\frac{1}{2\vep^2}\int_{(\T^\d)^2\setminus\triangle} \pa*{u(x)-u(y)}\cdot\nabla\g(x-y)df^{\otimes 2}(x,y)  
=\frac{1}{\vep^2}\int_{\T^\d}u\cdot \nab h^{f}d f.
\end{align}
By the fractional Leibniz rule and our assumption {$\ka>\d$}, the function $u\cdot\nab h^f \in H^{\frac{\ka}{2}}(\T^\d)$ and
\begin{align}
\|u\cdot\nab h^f\|_{H^{\frac{\ka}{2}}} &\le C\Big(\|u\|_{H^{\frac{\ka}{2}}} \|\nab h^f\|_{L^\infty} + \|u\|_{L^\infty} \|\nab h^f\|_{H^{\frac{\ka}{2}}}\Big) \nn\\
&\le C\Big(\|u\|_{H^{\frac{\ka}{2}}} + \|u\|_{L^\infty}\Big) \|\nab h^f\|_{H^{\frac{\ka}{2}}} \nn\\
&\le C\|u\|_{H^{\frac{\ka}{2}}}  \|f\|_{H^{-\frac{\ka}{2}}},
\end{align}
where the penultimate line follows from Sobolev embedding and the ultimate line follows from the assumption \eqref{eq:nabgHka} for $\g$ and another application of Sobolev embedding. The constant $C>0$ depends only on $\d,\ka$ and the constant in \eqref{eq:nabgHka}. 
It now follows from the $(H^{\frac{\ka}{2}})^* \cong H^{-\frac{\ka}{2}}$ duality that
\begin{align}\label{eq:Gronreg1}
\Big|\frac{1}{\vep^2}\int_{\T^\d}u\cdot \nab h^{f}d f\Big| \le \frac{C}{\vep^2}\|u\|_{H^\frac\ka2}\|f\|_{H^{-\frac{\ka}2}}^2.
\end{align}

{\bf The fourth term:} Instead of the coercivity estimate of \cref{lem:MPEcoer}, we use the $(H^{\frac{\ka}{2}})^* \cong H^{-\frac{\ka}{2}}$ duality, the assumption \eqref{eq:nabgHka} for $\g$, and the triangle inequality to estimate
\begin{align}
\Big|\int_{\T^\d}\pa*{\div h^{u\Uu}+h^{\p_t\Uu}}df\Big| &\le C\|f\|_{H^{-\frac{\ka}{2}}} \|\div h^{u\Uu}+h^{\p_t\Uu}\|_{H^{\frac{\ka}{2}}} \nn\\
&\le  C\|f\|_{H^{-\frac{\ka}{2}}} \Big( \|u\Uu\|_{H^{-\frac{\ka}{2}}} + \|(-\Delta)^{-\frac{1}{2}}\p_t \Uu\|_{H^{-\frac{\ka}{2}}}\Big) \nn\\
&\le \frac{C}{\vep^2}\|f\|_{H^{-\frac{\ka}{2}}}^2 + \vep^2 \Big( \|u\Uu\|_{H^{-\frac{\ka}{2}}} + \|(-\Delta)^{-\frac{1}{2}}\p_t \Uu\|_{H^{-\frac{\ka}{2}}}\Big)^2,  \label{eq:Gronreg1'}
\end{align}
where the final line follows from the elementary inequality $ab\le \frac12(\vep^{-2}a + \vep^2b)$. 
 
{\bf The fifth term:} By Cauchy-Schwarz in $i$,
\begin{align}
\frac1N\sum_{i=1}^N \Big|(v_i-u(x_i))\cdot(\nab\tl{h}^f)(x_i)\Big| &\le \Big(\frac1N\sum_{i=1}^N |v_i-u(x_i)|^2\Big)^{1/2} \Big(\frac1N\sum_{i=1}^N |(\nab\tl{h}^f)(x_i)|^2\Big)^{1/2} \nn\\
&\le \|\nab\tl{h}^{f}\|_{L^\infty} \Big(\frac1N\sum_{i=1}^N |v_i-u(x_i)|^2\Big)^{1/2} \nn\\
&\le C\|f\|_{H^{-\frac{\ka}{2}}}\Big(\frac1N\sum_{i=1}^N |v_i-u(x_i)|^2\Big)^{1/2} \nn\\
&\le \frac{C^2}{2}\|f\|_{H^{-\frac{\ka}{2}}}^2 + \frac{1}{2N}\sum_{i=1}^N |v_i-u(x_i)|^2, \label{eq:Gronreg2}
\end{align}
where the penultimate line follows from Sobolev embedding (by our assumption {$\ka>\d+2$}) and \eqref{eq:welec} and the ultimate line by $ab\le \frac12(a^2+b^2)$.

{\bf The sixth term:} We write
\begin{align}
\int_{\T^\d}u\cdot \nab \tl{h}^f df  &= \int_{\T^\d}u\cdot \nab\tl{h}^{f} (I-\Delta)^{\frac{\ka}{2}}\tl{h}^{f} \nn\\
&=\int_{\T^\d}\Big[(I-\Delta)^{\frac{\ka}{4}}(u\cdot\nab\tl{h}^{f}) - u\cdot\nab(I-\Delta)^{\frac{\ka}{4}}\tl{h}^{f}\Big](I-\Delta)^{\frac{\ka}{4}}\tl{h}^{f}\nn\\
&\ph + \int_{\T^\d}u\cdot\nab(I-\Delta)^{\frac{\ka}{4}}\tl{h}^{f} (I-\Delta)^{\frac{\ka}{4}}\tl{h}^{f}, \label{eq:gradwmuNmu}
\end{align}
where the final line follows from integrating by parts $(I-\Delta)^{\frac{\ka}{4}}$. To estimate the last line of \eqref{eq:gradwmuNmu}, we use the product rule to write $\nab(I-\Delta)^{\frac{\ka}{4}}\tl{h}^{f} (I-\Delta)^{\frac{\ka}{4}}\tl{h}^{f} = \frac12\nab\Big((I-\Delta)^{\frac{\ka}{4}}\tl{h}^{f}\Big)^2$ and then integrate by parts, leading to the final bound
\begin{align}\label{eq:gradwmuNmu1}
\Big|\int_{\T^\d}u\cdot\nab(I-\Delta)^{\frac{\ka}{4}}\tl{h}^{f} (I-\Delta)^{\frac{\ka}{4}}\tl{h}^{f}\Big| \le \frac12\|\div u\|_{L^\infty} \|(I-\Delta)^{\frac{\ka}{4}}\tl{h}^{f}\|_{L^2}^2 = \frac12\|\div u\|_{L^\infty} \|f\|_{H^{-\frac\ka2}}^2.
\end{align}
For the penultimate line of \eqref{eq:gradwmuNmu}, we recognize that the expression contained in the brackets is a commutator of Kato-Ponce type \cite{KP1988}. Applying Cauchy-Schwarz and \cite[Theorem 1.1]{Li2019} with $s=\frac{\ka}{2}$, $f= u$, $g=\nab\tl{h}^{f}$, and $p=2$, we obtain that
\begin{align}
&\int_{\T^\d}\Big|\Big[(I-\Delta)^{\frac{\ka}{4}}(u\cdot\nab\tl{h}^f) - u\cdot\nab(I-\Delta)^{\frac{\ka}{4}}\tl{h}^{f}\Big](I-\Delta)^{\frac{\ka}{4}}\tl{h}^{f}\Big| \nn\\
&\le \|(I-\Delta)^{\frac{\ka}{4}}(u\cdot\nab\tl{h}^{f}) - u\cdot\nab(I-\Delta)^{\frac{\ka}{4}}\tl{h}^{f}\|_{L^2}\|(I-\Delta)^{\frac{\ka}{4}}\tl{h}^{f}\|_{L^2} \nn\\
&\le C\Big(\|(I-\Delta)^{\frac{\ka}{4}}\nab u\|_{L^2} \|\nab\tl{h}^{f}\|_{L^\infty} + \|\nab u\|_{L^\infty}\|(I-\Delta)^{\frac{\ka-2}{4}}\tl{h}^{f}\|_{L^2} \Big)\|(I-\Delta)^{\frac{\ka}{4}}\tl{h}^{f}\|_{L^2} \nn\\
&\le C\|\nab u\|_{H^{\frac\ka2}} \|f\|_{H^{-\frac\ka2}}^2, \label{eq:gradwmuNmu2}
\end{align}
where in the final line, we have also used Sobolev embedding on $\|\nab\tl{h}^{f}\|_{L^\infty}$ and $\|\nab u\|_{L^\infty}$ and our assumption {$\ka>\d+2$}. 
Combining \eqref{eq:gradwmuNmu}, \eqref{eq:gradwmuNmu1}, \eqref{eq:gradwmuNmu2}, we conclude that
\begin{align}
\Big|\int_{\T^\d}u\cdot \nab\tl{h}^f  df\Big| \le C\|\nab u\|_{H^{\frac\ka2}} \|f\|_{H^{-\frac{\ka}{2}}}^2. \label{eq:Gronreg3}
\end{align}

{\bf The seventh term:} This term is analogous to the fourth term, and we ultimately find that
\begin{align}\label{eq:Gronreg4}
 \Big|\int_{\T^\d}\pa*{\div \tl{h}^{u\Uu}+\tl{h}^{\p_t\Uu}}df\Big| \le  \frac{C}{\vep^2}\|f\|_{H^{-\frac{\ka}{2}}}^2 + \vep^2 \Big( \|u\Uu\|_{H^{-\frac{\ka}{2}}} + \|(-\Delta)^{-\frac{1}{2}}\p_t \Uu\|_{H^{-\frac{\ka}{2}}}\Big)^2
\end{align}

Collecting the estimates \eqref{eq:Gronreg1}, \eqref{eq:Gronreg1'}, \eqref{eq:Gronreg2}, \eqref{eq:Gronreg3}, \eqref{eq:Gronreg4} and combining with the estimates for the remaining terms on the right-hand side of \eqref{eq:MEsmfIDreg} previously shown in \cref{ssec:MPgron}, we obtain that
\begin{multline}
\frac{d}{dt}\Hr_{N}(\uz_N,u) \le \frac{(\| (\nab u)_{sym,-} \|_{L^\infty}+\frac12-\ga)}{N}\sum_{i=1}^N |v_i-u(x_i)|^2 + \frac{C}{\vep^2}\|u\|_{H^{\frac\ka2}}\|f\|_{H^{-\frac{\ka}2}}^2 \\
+ \frac{C}{\vep^2}\|f\|_{H^{-\frac{\ka}{2}}}^2 + \vep^2 \Big( \|u\Uu\|_{H^{-\frac{\ka}{2}}} + \|(-\Delta)^{-\frac{1}{2}}\p_t \Uu\|_{H^{-\frac{\ka}{2}}}\Big)^2 + C\|f\|_{H^{-\frac{\ka}{2}}}^2 + \frac{C}{\vep^2}\|\nab u\|_{H^{\frac\ka2}}\|f\|_{H^{-\frac\ka2}}^2.
\end{multline}
Integrating and appealing to the Gr\"{o}nwall-Bellman lemma similar to before, we arrive at the desired \eqref{eq:SMFrigreg}.
\end{proof}

\subsection{Applicability to mean-field limits}\label{ssec:appregMF}

As mentioned above, this idea of adding an extra term (under the same assumption that $\ka>\d+2$ is determined by \eqref{eq:nabgHka}) to the modulated potential energy also works to prove mean-field limits (in the whole space or on the torus) of second-order monokinetic systems (i.e. \eqref{eq:NewODE} with $\vep=1$) and first-order systems
\begin{align}\label{eq:ODEfo}
\dot{x}_i^t = -\nab V(x_i^t) + \frac1N\sum_{1\le j\le N: j\ne i}\M\nab\g(x_i^t-x_j^t)
\end{align}
for a matrix $\M$ satisfying $\M\xi\cdot\xi\le 0$.

In the second-order case, one considers the total modulated energy
\begin{multline}
\mathsf{H}_N(Z_N^t,(\mu^t,u^t)) \coloneqq \frac{1}{2N}\sum_{i=1}^N |v_i^t-u^t(x_i^t)|^2 + \Fr_N(\ux_N^t,\mu^t)  \\
+\frac{1}{2}\int_{(\T^\d)^2}\w(x-y)d(\mu_N^t-\mu^t)^{\otimes 2}(x,y),
\end{multline}
where $Z_N^t$ is a solution of \eqref{eq:NewODE} with $\vep=1$ and $(\mu^t,u^t)$ is a monokinetic solution of the Vlasov equation, i.e. $f^t(x,v) = \mu^t(x)\delta(v-u^t(x))$ is a weak solution of the Vlasov equation. Here, we have re-inserted the diagonal (which is just a constant) in the definition of $\Fr_N$ with an abuse of notation. One can repeat the calculations in \cite[Appendix]{Serfaty2020} for this total modulated energy. The only difference is the new term (cf. \eqref{eq:dtwff})
\begin{align}
&\frac12\frac{d}{dt}\int_{(\T^\d)^2}\w(x-y)d(\mu_N^t - \mu^t)^{\otimes 2}(x,y) \nn\\
&= \frac1N\sum_{i=1}^N\int_{\T^\d}v_i^t\cdot\nab\w(x_i^t-y)d(\mu_N^t-\mu^t)(y) - \int_{\T^\d}u^t(x)\cdot\nab\w(x-y)d(\mu_N^t-\mu^t)(y)d\mu^t(x) \nn\\
&= \frac1N\sum_{i=1}^N(v_i^t-u^t(x_i^t)) \cdot \nab\tl{h}^{\mu_N^t-\mu^t}(x_i^t) + \int_{\T^\d}u^t\cdot\nab\tl{h}^{\mu_N^t-\mu^t} d(\mu_N^t-\mu^t).
\end{align}
One can then estimate all the terms similar to in the previous subsection, in particular replacing the commutator estimate used for Coulomb/Riesz interactions by the duality argument. 

In the first-order case, the kinetic term disappears and one instead considers the total modulated energy
\begin{align}
\mathsf{H}_N(X_N^t,\mu^t) \coloneqq \Fr_N(\XN^t,\mu^t) + \frac12\int_{(\T^\d)^2}\w(x-y)d(\mu_N^t-\mu^t)^{\otimes 2}(x,y),
\end{align}
where $X_N^t$ is a solution of \eqref{eq:ODEfo} and $\mu^t$ is the solution of the mean-field PDE associated to \eqref{eq:ODEfo}.  Repeating the calculations in \cite{Serfaty2020},  the only difference is the new term (assuming for simplicity that $\nab\g(0)=0$)
\begin{align}
&\frac12\frac{d}{dt}\int_{(\T^\d)^2}\w(x-y)d(\mu_N^t - \mu^t)^{\otimes 2}(x,y) \nn\\
&=\frac1N\sum_{i=1}^N \int_{\T^\d}\Big(-\nab V(x_i^t) + \frac1N\sum_{j=1}^N\M\nab\g(x_i^t-x_j^t)\Big)\cdot \nab \w(x_i^t-y)d(\mu_N^t-\mu^t)(y) \nn\\
&\ph+ \int_{\T^\d}\Big(\nab V - \M\nab(\g\ast\mu^t)\Big)\cdot\nab\w(x-y)d(\mu_N^t-\mu^t)(y)d\mu^t(x) \nn\\
&= -\int_{\T^\d}\nab V\cdot \tl{h}^{\mu_N^t-\mu^t}d(\mu_N^t-\mu^t) +\frac1N\sum_{i=1}^N \M\nab\g\ast(\mu_N^t-\mu^t)(x_i^t) \cdot \nab\tl{h}^{\mu_N^t-\mu^t}(x_i^t) \nn\\
&\ph+ \int_{\T^\d}\M\nab(\g\ast \mu^t)\cdot \nab\tl{h}^{\mu_N^t-\mu^t} d(\mu_N^t-\mu^t).
\end{align}
The first and third terms on the right-hand side of the second equality may be handled by the same duality argument used in the previous subsection. For the second term, we argue
\begin{align}
\frac1N\sum_{i=1}^N |\M\nab\g\ast(\mu_N^t-\mu^t)(x_i^t) \cdot \nab\tl{h}^{\mu_N^t-\mu^t}(x_i^t)| &\le C|\M|\|\nab\g\ast(\mu_N^t-\mu^t)\|_{L^\infty}\|\nab\tl{h}^{\mu_N^t-\mu^t}\|_{L^\infty} \nn\\
&\le  C|\M| \|\nab\g\ast(\mu_N^t-\mu^t)\|_{H^\frac\ka2} \|\mu_N^t-\mu^t\|_{H^{-\frac{\ka}2}},
\end{align}
where the final line is by Sobolev embedding and the assumption $\ka>\d+2$. The first factor on the last line may be bounded in terms of $\|\mu_N^t-\mu^t\|_{H^{-\frac\ka2}}$  using the assumption \eqref{eq:nabgHka} on $\g$.

\section{Regularity for the obstacle problem}\label{sec:appOP}
In this section, we review some classical regularity {and lift-off rates (which are intimately connected)} for solutions of the obstacle problem near the free boundary, as well as some new results that are used in the proof of \cref{lem:ugradVext}. 

Throughout this subsection, $u$ will denote either a solution of the \emph{local obstacle problem}
\begin{align}\label{eq:locOP}
\min\{(-\Delta)^s u, u-\psi\} = 0 \quad \text{in} \ B_1\subset\R^\d
\end{align}
or the \emph{global obstacle problem}
\begin{align}\label{eq:globOP}
\min\{(-\Delta)^s u, u-\psi\} = 0 \quad \text{in} \ \R^\d.
\end{align}
In all cases, $s\in (0,1)$. 

The optimal regularity of solutions was first established by Caffarelli {et al.} in \cite{CSS2008} for $\psi\in C^{2,1}$ obstacles. The regularity assumption on the obstacle was subsequently improved to $\psi\in C^{1,s+\delta}$ in \cite{CdSS2017}. In the same work \cite{CSS2008}, the regularity of free boundaries was established and then later improved by various authors in \cite{CdSS2017,rOS2017}. We summarize these results with the following proposition.

\begin{prop}\label{prop:OPreg}
Let $\psi \in C^{1,s+\delta}(B_1)$, for some $\delta>0$, and $u$ be any viscosity solution of \eqref{eq:locOP}. Then $u\in C^{1,s}(B_{1/2})$.
Suppose now that $\psi \in C^{1+2s+\delta}$ for some $\delta>0$ and $u$ is a solution of \eqref{eq:locOP}. Then for every free boundary point $x_*\in \p\{u>\psi\} \cap B_{1/2}$, we have the following dichotomy:
\begin{enumerate}[(i)]
\item\label{item:regpt} $x_*$ is a \emph{regular point},
\begin{align}\label{eq:regptasy}
(u-\psi)(x) = c_{x_*}d^{1+s}(x) +  O(|x-x_*|^{1+s+\al}),
\end{align}
with $c_{x_*}>0$, where $d$ is the distance to the contact set $\{u=\psi\}$. In which case, the free boundary is $C^{1,\al}$ in a neighborhood of $x_*$.
\item\label{item:singpt} $x_*$ is a \emph{degenerate/singular point},
\begin{align}
(u-\psi)(x) =  O(|x-x_*|^{1+s+\al}).
\end{align}
\end{enumerate}
\end{prop}

\begin{remark}
Note that 
\begin{align}
\lim_{\substack{x\rightarrow x_* \\ x\in B_1}} \frac{u-\psi}{d^{1+s}},
\end{align}
always exists and is strictly positive if $x_*$ is a regular point and zero if $x_*$ is a degenerate/singular point.
\end{remark}

\begin{remark}\label{rem:OPreg}
One actually has the quantitative local $C^{1,s}$ bound
\begin{align}\label{eq:OPreg}
\|u\|_{C^{1,s}(B_{1/2})} \le C(\|\psi\|_{C^{1,s+\delta}(B_1)} + \|u\|_{L^\infty(\R^\d)}),
\end{align}
where the constant $C>0$ depends only on $\d,s,\delta$. In fact, the dependence on $\|u\|_{L^\infty(\R^\d)}<\infty$ can be relaxed to a local estimate $L^\infty$ estimate for the $s$-harmonic extension $\tl{u}$ in $\R^{\d+1}$ of $u$ in the form of $\|\tl{u}\|_{L^\infty(\tl{B}_1)}$, where $\tl{B}_1$ is the ball of radius $1$ in $\R^{\d+1}$. This is useful because  $u=h^{\mu_V}\notin L^\infty(\R^\d)$ if $\d=1$ and $\s\le 0$, as it grows in magnitude like $|x|^{-\s}$ at infinity. However, its $(\frac{\d-\s}{2})$-harmonic extension is given by simply extending $h^{\mu_V}$ to $\R^{\d+1}$ through the radial symmetry of $\g$, and this is in $L_{loc}^\infty(\R^{\d+1})$, assuming, say, that $\mu_V$ is bounded.
\end{remark}


We next reproduce below two results adapted from \cite{ArO2020}. The first is a sharp (higher) regularity estimate for regular points of the free boundary, which is adapted from \cite[Theorem 1.2]{ArO2020}. The second is a boundary Schauder-type estimate for solutions of nonlocal elliptic equations in $C^{k,\al}$ domains, which is adapted from \cite[Theorem 1.4]{ArO2020}.

\begin{prop}\label{lem:ArO1}
Let $u$ be any solution of \eqref{eq:globOP} with obstacle $\psi$ such that $\{\psi>0\}$ is bounded. Let $\theta>2$ be such that both  $\theta,\theta\pm s$ are noninteger, and assume that $\psi \in C^{\theta+s}$. Then the free boundary is $C^\theta$ in a neighborhood $B_r(x_*)$ of any regular point $x_* \in \p\{u>\psi\}$ and the $C^\theta$  norm only depends on $\d,s,\ga,\theta,r$. 
\end{prop}

\begin{prop}\label{lem:ArO2}
Let $\be>s$ be such that both $\be, \be\pm s$ are noninteger. Let $\Omega\subset\R^\d$ be a bounded $C^{\be+1}$ domain. Then there exists a constant $C>0$ depending only on $\d,s,\be,\Omega$, such that for any solution $u\in L^\infty(\R^\d)$ of
\begin{align}\label{eq:SchaBVP}
\begin{cases}
(-\Delta)^s u = f & \quad \text{in} \ \Omega\cap B_1\\
u = 0 & \quad \text{in} \ B_1\setminus\Omega
\end{cases}
\end{align}
with $f\in C^{\be-s}(\ol\Omega)$, it holds that
\begin{align}\label{eq:ArO2}
\|\frac{u}{d^s}\|_{C^\be({\Omega}\cap B_1)} \le C(\|f\|_{C^{\be-s}(\ol\Omega)} + \|u\|_{L^\infty(\R^\d)}),
\end{align}
{where $d(x)\coloneqq \dist(x,\Omega^c)$.}
\end{prop}

\begin{remark}\label{rem:ArO2}
The boundedness assumption on $\Omega$ is not actually necessary, since we may always replace $\Omega$  by $B_2\cap \Omega$ without changing \eqref{eq:SchaBVP}. Similarly, the boundedness assumption on $\{\psi>0\}$ is not necessary because a solution $u$ of \eqref{eq:globOP} solves the local obstacle problem \eqref{eq:locOP} with an obstacle that is localized to vanish outside $B_2$.

Furthermore, the global assumption ${u\in L^\infty(\R^\d)}$ may be removed at the cost of shrinking the ball on the left-hand side {of \eqref{eq:ArO2}}. More precisely, let $\chi$ be a smooth cutoff which is identically one in $B_{3/4}$ and zero outside $B_1$. Set $\tl u = \chi u$. Then $\tl{u}$ solves
\begin{align}\label{eq:SchaBVP'}
\begin{cases}
(-\Delta)^s\tl{u} = \tl{f} & \quad \text{in} \ \Omega\cap B_{1/2}\\
\tl{u} = 0 & \quad \text{in} \ B_{1/2}\setminus\Omega
\end{cases}
\end{align}
with $\tl{f} = f - (-\Delta)^s(\chi^c u)$, where $\chi^c \coloneqq 1-\chi$. By \cref{lem:ArO2} (rescaling $B_{1/2}$ to $B_1$ and replacing $\Omega$ by $\Omega\cap B_{5/8}$), we have that
\begin{align}
\|\frac{u}{d^s}\|_{C^\be({\Omega}\cap B_{1/2})} &=  \|\frac{\tl{u}}{d^s}\|_{C^\be({\Omega}\cap B_{1/2})} \le C(\|\tl{f}\|_{C^{\be-s}(\ol\Omega)} + \|\tl{u}\|_{L^\infty(\R^\d)}) \nn\\
&\le C(\|f\|_{C^{\be-s}(\ol\Omega)} + \|(-\Delta)^s(\chi^c u)\|_{C^{\be-s}(\ol{\Omega \cap B_{5/8}})} + \|u\|_{L^\infty(B_1)}),
\end{align}
where we have also used the triangle inequality in the last line. Observe that $\chi^c \equiv 0$ in $B_{3/4}$. Hence, for $x\in \ol{\Omega \cap B_{5/8}}$,
\begin{align}\label{eq:Deltaschic}
(-\Delta)^s(\chi^c u)(x) = -c\int_{y\notin B_{3/4}} \frac{\chi^c(y)u(y)}{|x-y|^{\d+2s}}.
\end{align}
This function is evidently smooth, assuming that
\begin{align}
\|u\|_{L_{\d+2s}^1(\R^\d)} \coloneqq \int_{\R^\d}\frac{|u(y)|}{(1+|y|)^{\d+2s}}<\infty.
\end{align}
{This moment assumption is always satisfied for $u=h^{\mu_V}$ and $s=\frac{\d-\s}{2}$. Indeed, $h^{\mu_V}$ is controlled by $\|\mu_V\|_{L^\infty}$ if $\d\ge 2$ and $\s>\d-2\ge 0$. If $\d=1$ and $\s \in (-1,0]$, then $h^{\mu_V}$ grows like $|x|^{-\s}$, while $\d+2s = 2-\s$.}
\end{remark}

For the purposes of proving \cref{lem:ugradVext}, we need the following result for the tangential and normal derivatives of $u-\psi$, a proof of which does not seem present in the literature. This is an application of the preceding propositions.

\begin{prop}\label{prop:OPnt}
Let $\theta>2$ such that $\theta,\theta\pm s$ are not integers. Let $u$ be any solution of \eqref{eq:globOP} with obstacle $\psi \in C_{loc}^{\theta+s+\alpha}$, for any $\alpha>0$, such that $\|u\|_{L_{\d+1+2s}^1(\R^\d)}<\infty$. Then given any regular point $x_*$, there exist radii $0<r'<r$ such that the free boundary is $C^{\theta}$ in $B_r(x_*)$ and each point in $x\in B_{r'}(x_*)$ has a unique nearest-point projection {$x_\circ$ onto $\{u=\psi\}$ such that $x_\circ\in B_{r}(x_*)$.} Moreover, if $\nu_\circ,\tau_\circ$ are respectively normal and tangent directions at $x_\circ$, then there exists $C>0$ such that for any $x\in B_{r'}(x_*)$,
\begin{align}
|\p_{\nu_{\circ}}(u-\psi)(x)|&\le C d^s(x),  \label{eq:OPntn}\\
|\p_{\tau_\circ}(u-\psi)(x)| &\le Cd^{s+1}(x), \label{eq:OPntt}
\end{align}
where $C>0$ is independent of $\nu_\circ,\tau_\circ$. 
\end{prop}

\begin{proof}

Let $x_*$ be a regular point. If $\psi \in C^{\theta+s}$, for $\theta>2$ such that $\theta,\theta\pm s$ are not integers, then by \cref{lem:ArO1}, the free boundary is $C^\theta$ in a neighborhood of $B_r(x_*)$.

Setting $w\coloneqq \p(u-\psi)$ for any partial derivative $\p$, the fact that $u,\psi$ are $C^1$ implies that $w$ vanishes on $\{u=\psi\}$, hence is a solution to
\begin{align}
\begin{cases}
(-\Delta)^s w= {-}(- \Delta)^s\p\psi, & \quad \text{in} \ \{u>\psi\} \\
w =0, & \quad \text{in}  \ \{u=\psi\}.
\end{cases}
\end{align}
Let $\be \coloneqq \theta-1$. Then $\be>s$ and $\be,\be\pm s$ are not integers.  Moreover, $\Omega \coloneqq \{u>\psi\} \cap B_{r}(x_*)$ is a bounded $C^{\be + 1}$ domain by the previous paragraph. As in \cref{rem:ArO2}, let $\chi$ be a cutoff as above, and consider $\tl{w} \coloneqq \chi w$, which solves
\begin{align}
\begin{cases}
(-\Delta)^s\tl{w}= -(- \Delta)^s(\chi^c {\p u}) {-} (- \Delta)^s(\chi\p\psi), & \quad \text{in} \ \{u>\psi\} \cap B_{r/2}(x_*) \\
\tl{w} =0, & \quad \text{in}  \  B_{r/2}(x_*) \cap \{u=\psi\}.
\end{cases}
\end{align}
Hence, we may apply \cref{lem:ArO2} and \cref{rem:ArO2} to obtain
\begin{align}\label{eq:wdsCth}
&\|\frac{w}{d^s}\|_{C^{\theta-1}(\{u>\psi\}\cap B_{r/2}(x_*))} \nn\\
&\le C\Big( \|(-\Delta)^s(\chi\p\psi)\|_{C^{\theta-1-s}(\ol{\{u>\psi\}\cap B_{r}(x_*)})} +  \|(-\Delta)^s(\chi^c\p { u})\|_{C^{\be-1-s}(\ol{\Omega\cap B_{5r/8}(x_*)})} \nn\\
&\ph+ \|\p(u-\psi)\|_{L^\infty(B_{r}(x_*))}\Big).
\end{align}
Remark that the constant $C$ is independent of the choice of $\p$. From the definition of the fractional Laplacian, it is straightforward to check that
\begin{align}
 \|(-\Delta)^s(\chi\p\psi)\|_{C^{\theta-1-s}(\ol{\{u>\psi\}\cap B_{r}(x_*)})} \le C \|\psi\|_{C^{\theta+s+\alpha}(B_{r}(x_*))}
\end{align}
for any $\alpha>0$. By triangle inequality,
\begin{align}
\|\p(u-\psi)\|_{L^\infty(B_{2r}(x_*))} \le \|\p u\|_{L^\infty(B_{2r}(x_*))} + \|\p\psi\|_{L^\infty(B_{2r}(x_*))},
\end{align}
and the term $\|\nab u\|_{L^\infty(B_{2r}(x_*))}$ is finite by \cref{prop:OPreg}. Finally, for any test function $f$, using \eqref{eq:Deltaschic} with $u$ replaced by $\p f$, we can integrate by parts, using that $|x-y|^{\d+2s}$ is smooth for $x  \in \ol{B_{5r/8}}$ and $y\notin B_{3r/4}(x_*)$, to obtain
\begin{align}
(-\Delta)^s(\chi^c \p f)(x) = c\int_{\R^\d}\frac{\p\chi^c(y)f(y)}{|x-y|^{\d+2s}} -(\d+2s) c\int_{\R^\d} \frac{\p|x-y| \chi^c(y)f(y)}{|x-y|^{\d+2s+1}}.
\end{align}
Since $\p\chi^c(y)$ is zero for $y\in B_{3r/4}(x_*)$ and $y\notin B_{5r/8}(x_*)$, the first term on the right-hand side is smooth if $f$ is in $L_1(B_{5r/8}(x_*))$. The second term is smooth if $\int_{\R^{\d+\k}}\frac{|f(y)|}{(1+|y|)^{\d+2s+1}}<\infty$. Applying this result to {$f=u$}, we conclude that
\begin{align}
\|(-\Delta)^s(\chi^c\p {u})\|_{C^{\be-1-s}(\ol{\Omega\cap B_{5/8r}(x_*)})}  <\infty.
\end{align}
Recalling our starting point \eqref{eq:wdsCth}, we have shown that there is a $C>0$ independent of $\p$, such that
\begin{align}\label{eq:wdsCth'}
\forall x\in B_{r/2}(x_*), \qquad \|\frac{w}{d^s}\|_{C^{\theta-1}(\{u>\psi\}\cap B_{r/2}(x_*))} \le C,
\end{align}
{Recalling that $w=\p(u-\psi)$ for an arbitrary partial derivative $\p$,} the preceding automatically implies that given any free boundary point $x_\circ$ with normal vector $\nu_{\circ}$,  $\p_{\nu_\circ}(u-\psi)$ satisfies
\begin{align}
\forall x\in B_{r/2}(x_*), \qquad |\p_{\nu_{\circ}}(u-\psi)(x)|\le C d^s(x).
\end{align}
This establishes \eqref{eq:OPntn}.

For the tangential derivatives, we first recall that there exists $r'\in (0,r/2)$ such that every point $x\in B_{r'}(x_*)$ has a unique nearest-point projection $x_\circ$ onto $\{u=\psi\}$ in $B_{2r'}(x_*)$. Let $\tau_\circ$ be a tangent vector at $x_\circ$. Then since $\p_{\tau_\circ}d (x) = 0$, it follows that
\begin{align}
\forall x\in B_{r'}(x_*), \qquad \p_{\tau_\circ}\Big(\frac{w}{d^s}\Big)(x) = \frac{\p_{\tau_\circ}\p(u-\psi)(x)}{d^s(x)}.
\end{align}
Here, we are implicitly using that $\theta-1>1$ (by assumption that $\theta>2$) and so $\frac{w}{d^s}$ is at least $C^1$. Moreover, by \eqref{eq:wdsCth'},
\begin{align}
\forall x\in B_{r'}(x_*), \qquad  |{\p_{\tau_\circ}\p(u-\psi)(x)}|\le C d^s(x),
\end{align}
where again $C$ is independent of the choice $\p$. Now since $\p_{\tau_\circ}(u-\psi)(x_\circ)=0$, it follows from the mean value theorem that for any $x\in B_{r'}(x_*)$, 
\begin{align}
|\p_{\tau_\circ}(u-\psi)(x)| = |\p_{\tau_\circ}(u-\psi)(x) - \p_{\tau_\circ}(u-\psi)(x_\circ)| \le Cd^{s+1}(x).
\end{align}
This establishes \eqref{eq:OPntt}, and the proof of the proposition is complete. 
\end{proof}

\begin{remark}
One can easily extract a quantitative bound in terms norms of $u,\psi$ for the constant $C$ in \cref{prop:OPnt} by using \cref{rem:OPreg}.
\end{remark}


Using \cref{prop:OPnt}, we have the following refinement of the asymptotics \eqref{eq:regptasy}.

\begin{prop}\label{prop:OPreglb}
Impose the same assumptions as \cref{prop:OPnt}, and let $x_*,r$ be as above. Then there exists $c>0$ such that
\begin{align}
\forall x\in \{u>\psi\} \cap B_{r}(x_*), \qquad (u-\psi)(x) \ge cd^{1+s}(x).
\end{align}
\end{prop}
\begin{proof}
Indeed, let $x_*,r,r'$ be as in the statement of \cref{prop:OPnt}. By translation and rotation, we may assume without loss of generality that $x_* = 0$ and the free boundary in $B_r(0)$ is given by the $C^\theta$ graph
\begin{align}
\{u\ge \psi\} \cap B_{r}(0) = \{x = (x^1,\ldots,x^\d) : x^\d \ge \Theta(x^1,\ldots,x^{\d-1}) , \ x\in B_r(0)\}
\end{align}
with $\Theta(0) =0$, $\nab\Theta(0)=0$, and inward normal vector $e_\d \coloneqq (0,\ldots,0,1)$ at $0$. Let us abbreviate $x_\perp \coloneqq (x^1,\ldots,x^{\d-1})$. As noted in \cite{ArO2020} (see the proof of Theorem 1.2), there exists $c>0$ such that
\begin{align}\label{eq:pdupsi}
\p_\d(u-\psi)(x) \ge c d^s(x), \qquad \forall x\in \{u>\psi\} \cap B_r(0).
\end{align}
Thus, for $x=(x_\perp, x^\d) \in \{u>\psi\} \cap B_r(0)$, the fundamental theorem of calculus implies that
\begin{align}
(u-\psi)(x) &= (u-\psi)(x)- (u-\psi)(x_\perp, \Theta(x_\perp)) \nn\\
& = (x^\d - \Theta(x_\perp))\int_0^1 \p_{\d}(u-\psi)(x_\perp, \Theta(x_\perp) + t(x^{\d}-\Theta(x_\perp))) dt \nn\\
&\ge cd^{s+1}(x),
\end{align}
where the final line follows from \eqref{eq:pdupsi} and the fact that $(x^\d - \Theta(x_\perp)) \ge d(x)$ by definition of the distance function. This completes the proof.
\end{proof}


We now pay our debt to the reader by giving the proof of \cref{lem:ugradVext}.

\begin{proof}[Proof of \cref{lem:ugradVext}]
We only consider the fractional case $\s\ne \d-2$, as the Coulomb case $\s=\d-2$ was previously treated in \cite[Lemma 3.2]{BCGM2015}.

We recall that $u=h^{\mu_V}$ is a solution of the obstacle problem \eqref{eq:OPfl} for exponent ${s}=\frac{\d-\s}{2}$ and  obstacle $\varphi=c-V$. By assumption, every free boundary point $x_*\in \p\Sigma = \p\{\zeta =0\}$ is a regular point. Hence, letting $0<r_{x_*}' < r_{x_*}$ be as in \Cref{prop:OPnt,prop:OPreglb}, the family $\{B_{r'_{x_*}}(x_*)\}_{x_*\in \p\Sigma}$ forms an open cover of $\p\Sigma$. Since $\p\Sigma$ is compact, there exists a finite subcover corresponding to points $x_{*,1},\ldots,x_{*,n}$.

By \cref{prop:OPreg} and our assumption for $\varphi$, $\zeta\in C_{loc}^{1,\frac{\d-\s}{2}}$. Since $\nab\zeta = 0$ in $\Sigma$, the left-hand side of \eqref{eq:ugradzeta} trivially holds for $x\in\Sigma$. Letting $d$ denote the distance to the contact set $\{\zeta =0\}$, it follows that there exists $\delta>0$ such that
\begin{align}
\inf\{ d(x) : x \in \{\zeta > 0\} \setminus \bigcup_{i=1}^n B_{r'_{x_*,1}}(x_{*,1})\} \ge \delta.
\end{align}
Since $\zeta>0$ in $\R^\d\setminus\Sigma$ and continuous, it follows from the extreme value theorem that there exists $c_0>0$ such that
\begin{align}\label{eq:ugradex1'}
\min_{x : \delta \le d(x) \le 2\diam\Sigma} \zeta(x) \ge c_0.
\end{align}
Since the left-hand side of \eqref{eq:ugradzeta} is zero for $x$ such that $d(x)\ge 2\diam\Sigma$, the inequality \eqref{eq:ugradzeta} trivially holds for such $x$. For {$\delta\le d(x)\le 2\diam \Sigma$}, we crudely bound
\begin{align}\label{eq:ugradex2'}
|v(x)\cdot \nab\zeta(x)| &\le \|v\|_{L^\infty}\|\nab\zeta\|_{L^\infty(\tl\Sigma)}  \le \|v\|_{L^\infty}\Big(\|\nab V\|_{L^\infty(\tl\Sigma)} + \|\nab h^{\mu_V}\|_{L^\infty(\tl\Sigma)}\Big).
\end{align}
where {$\tl\Sigma \coloneqq \{x: d(x)\le 2\diam\Sigma\}$ and the final inequality is by the triangle inequality}. The second term is finite by \cref{prop:OPreg}. So combining \eqref{eq:ugradex1'}, \eqref{eq:ugradex2'}, we find that
\begin{align}\label{eq:ugradex3'}
\delta\le d(x)\le 2\diam\Sigma, \qquad |v(x)\cdot\nab\zeta(x)| \le \frac{\|v\|_{L^\infty}\Big(\|\nab V\|_{L^\infty(\Sigma_\delta)} + \|\nab h^{\mu_V}\|_{L^\infty(\Sigma_\delta)}\Big)}{c_0} \zeta(x).
\end{align}
{Since $\{d(x)\le\delta\} \subset \bigcup_{i=1}^n B_{r'_{x_{*,i}}}(x_{*,i})$,} it suffices now to establish \eqref{eq:ugradzeta} when $x\in \bigcup_{i=1}^n B_{r'_{x_{*,i}}}(x_{*,i})$.



Dropping the $i$ subscript, let $x\in B_{r'}(x_*) \cap \{\zeta>0\}$. Let $\nu_\circ$ denote the normal vector at the nearest-point projection $x_\circ$ onto $\{\zeta=0\}$. Without loss of generality (rotating and translating if necessary), we may assume that $\nu_\circ = e_{\d} = (0,\ldots,0,1)$. Hence, $e_{1},\ldots,e_{\d-1}$ are tangent to $\{\zeta=0\}$ at $x_\circ$. We write
\begin{align}\label{eq:ugradex1}
v(x)\cdot\nab\zeta(x) = \sum_{j=1}^\d v^j(x)\p_j\zeta(x).
\end{align}
Using the no-flux condition $v^\d(x_\circ) = 0$, it follows from the mean-value theorem and the estimate \eqref{eq:OPntn} of \cref{prop:OPnt} {applied with $s= \frac{\d-\s}{2}$} that
\begin{align}\label{eq:ugradex2}
|v^\d(x)\p_\d\zeta(x)| \le |v^\d(x)-v^\d(x_\circ)| |\p_\d\zeta(x)| \le C\|\nab v^\d\|_{L^\infty(B_{r'}(x_*))}d^{1+\frac{\d-\s}{2}}(x),
\end{align}
where we have implicitly used that $|x-x_\circ| = d(x)$. Using the estimate \eqref{eq:OPntt} of \cref{prop:OPnt}, we also have
\begin{align}\label{eq:ugradex3}
\sum_{j=1}^{\d-1} |v^j(x)\p_j\zeta(x)| \le C\sum_{j=1}^{\d-1} \|v^j\|_{L^\infty(B_{r'}(x_*))} d^{1+\frac{\d-\s}{2}}(x) \le C\|v\|_{L^\infty(B_{r'}(x_*))} d^{1+\frac{\d-\s}{2}}(x).
\end{align}
Combining \eqref{eq:ugradex1}, \eqref{eq:ugradex2}, \eqref{eq:ugradex3}, we conclude that there is a $C>0$, independent of $v$, such that
\begin{align}\label{eq:ugradex4}
\forall x\in B_{r'}(x_*) \cap \{\zeta >0\}, \qquad |v(x)\cdot\nab\zeta(x)| \le C\|v\|_{W^{1,\infty}(B_{r'}(x_*))} d^{1+\frac{\d-\s}{2}}(x).
\end{align}
On the other hand, by \cref{prop:OPreglb}, there exists $c>0$ such that
\begin{align}\label{eq:ugradex5}
\forall x\in B_{r'}(x_*) \cap \{\zeta >0\}, \qquad \zeta(x) \ge c d^{1+\frac{\d-\s}{2}}(x).
\end{align}
Combining \eqref{eq:ugradex4}, \eqref{eq:ugradex5} {and setting $A\coloneqq C/c$}, we find that
\begin{align}
\forall x\in B_{r'}(x_*) \cap \{\zeta >0\}, \qquad |v(x)\cdot\nab\zeta(x)| \le A\|v\|_{W^{1,\infty}(B_{r'}(x_*))}\zeta(x).
\end{align}

Applying \eqref{eq:ugradex4} for each ball $B_{r_{x_{*,i}}'}(x_{*,i})$, we obtain finitely many constants $A_1,\ldots,A_n>0$, so that setting $A_{\max} = \max(A_1,\ldots,A_n)$, we conclude
\begin{align}\label{eq:ugradex6}
\forall x\in \bigcup_{i=1}^n B_{r_{x_{*,i}}'}(x_{*,i}) \cap \{\zeta>0\}, \qquad |v(x)\cdot\nab\zeta(x)| \le A_{\max}\|v\|_{W^{1,\infty}} \zeta(x).
\end{align}

Combining \eqref{eq:ugradex6}, \eqref{eq:ugradex3'}, the proof of the lemma is complete.
\end{proof}

\bibliographystyle{alpha}\bibliography{../../MASTER}
\end{document}